\documentclass[12pt,a4paper]{article}
\usepackage[utf8]{inputenc}
\usepackage[T1]{fontenc}
\usepackage{lmodern} 
\usepackage[english]{babel}
\usepackage{pstricks,pst-all} 
\usepackage{tikz}
\usetikzlibrary{arrows,chains,matrix,positioning,scopes}
\usepackage{amsmath,amsthm}
\usepackage{pifont}
\usepackage{bbm} 
\usepackage{amssymb}
\usepackage{mathabx}
\usepackage{amsfonts}
\usepackage{mathrsfs}

\usepackage[a4paper]{geometry}
\usepackage{stmaryrd}
\usepackage{dsfont}
\usepackage{graphicx}
\usepackage[active]{srcltx} 
\newcommand{\be}{\begin{enumerate}}
\newcommand{\ee}{\end{enumerate}}
\usepackage{enumerate}

\newcommand{\conv}[2]{\operatorname*{\longrightarrow}_{#1 \rightarrow #2}}
\newcommand{\convw}[2]{\operatorname*{\longrightharpoonup}_{#1 \rightarrow #2}}
\newcommand{\prods}[3]{\left\langle #1, #2\right\rangle_{#3}}
\newcommand{\ds}{\displaystyle}
\newcommand{\R}{\mathbb{R}}

\newcommand{\din}{\dot{\,\in\,}}
\newcommand{\ddin}{\ddot{\,\in\,}}
\newcommand{\D}{\mathscr{D}}
\newcommand{\Dd}{\mathscr{D}_{\textnormal{div}}}
\newcommand{\Pp}{\mathbb{P}}
\newcommand{\Omegah}{\hat{\Omega}}
\newcommand{\gammah}{\hat{\gamma}}

\newcommand{\Ss}{\textnormal{S}}

\newcommand{\Id}{\textnormal{Id}}
\newcommand{\Cc}{\textnormal{C}}

\newcommand{\J}{\textnormal{J}}

\newcommand{\B}{\textnormal{B}}
\newcommand{\dd}{\mathrm{d}}
\newcommand{\N}{\mathbb{N}}

\newcommand{\ffi}{\varphi}
\newcommand{\Ker}{\textnormal{Ker}}

\newcommand{\uu}{\mathbf{u}}
\newcommand{\ww}{\mathbf{w}}
\newcommand{\vv}{\mathbf{v}}

\newcommand{\yy}{\mathbf{y}}
\newcommand{\hh}{\mathbf{h}}
\newcommand{\xx}{\mathbf{x}}
\newcommand{\zz}{\mathbf{z}}

\newcommand{\ppsi}{\boldsymbol{\psi}}
\newcommand{\Pphi}{\boldsymbol{\Phi}}
\newcommand{\fffi}{\boldsymbol{\ffi}}

\newcommand{\nn}{\mathbf{n}}

\newcommand{\Hh}{\textnormal{H}}
\newcommand{\W}{\textnormal{W}}

\renewcommand{\H}{\textnormal{H}}

\newcommand{\Ldiv}{\textnormal{L}_{\div}}

\newcommand{\Ld}{\textnormal{L}}

\newcommand{\Ldivoo}{\textnormal{L}_{\div,0}}

\newcommand{\K}{\textnormal{K}}
\newcommand{\Wl}{\textnormal{W}_\textnormal{loc}}
\renewcommand{\W}{\textnormal{W}}
\newcommand{\Ll}{\textnormal{L}_\textnormal{loc}}
\renewcommand{\P}{\textnormal{P}}
\renewcommand{\div}{\textnormal{div}}

\newcommand{\Ddiv}{\mathscr{D}_{\div}}

\newcommand{\ep}{\varepsilon}

\newtheoremstyle{RemarkStyle}{3mm}{3mm}{\em}{}{\sf}{ :}{\newline}{}
\newtheoremstyle{ClassikStyle}{3mm}{3mm}{\em}{}{\bf}{ :}{2mm}{}
\newtheorem{fact}{Fact}[section]
\def\longrightharpoonup{\relbar\joinrel\rightharpoonup}
\theoremstyle{RemarkStyle}

\theoremstyle{ClassikStyle}
\newtheorem{theorem}{Theorem}
\newtheorem{lemme}{Lemme}[section]
\newtheorem{lem}{Lemma}
\newtheorem{proposition}{Proposition}
\newtheorem{coro}{Corollary}
\newtheorem{rmk}{Remark}[section]

\newtheorem{assumption}{Assumption}

\title{Some variants of the classical Aubin-Lions Lemma}
\date{}

\begin{document}
\bibliographystyle{abbrv}
\maketitle
\centerline{\scshape A. Moussa\footnote{Sorbonne Universités, UPMC Univ Paris 06, UMR 7598, Laboratoire Jacques-Louis Lions, F-75005, Paris, France}\footnote{CNRS, UMR 7598, Laboratoire Jacques-Louis Lions, F-75005, Paris, France}}
\medskip
{\footnotesize
\centerline{E-mail : moussa@ann.jussieu.fr}
\bigskip
\centerline{\today}
\medskip
\abstract{This paper explores two generalizations of the classical Aubin-Lions Lemma. First we give a sufficient condition to commute weak limit and multiplication of two functions. We deduce from this criteria a compactness Theorem for degenerate parabolic equations. Secondly, we state and prove a compactness Theorem for non-cylindrical domains, including the case of dual estimates involving only divergence-free test functions.}

\vspace{4mm}


\section{Introduction}
\subsection{Aubin-Lions Lemma and beyond}\label{subsec:beyond}
In the study of nonlinear evolution equations, the Aubin-Lions lemma is a powerful tool allowing to handle the nonlinear terms, when dealing with an approximation process or asymptotic limit. The standard statement gives sufficient conditions on a sequence of functions $(u_n)_n$ of two variables $(t,x)$ (time variable $t$ belongs to some interval $I$, space variable $x$ belongs to some open bounded set $\Omega$) bounded in $\Ld^p(I;B)$ where $B$ is some Banach space of functions defined on $\Omega$. More precisely, if  
\begin{itemize}
\item[(i)] $(u_n)_n$ is bounded in $\Ld^p(I;X)$;
\item[(ii)] $(\partial_t u_n)_n$ is bounded in $\Ld^r(I;Y)$;
\item[(iii)] $X$ embeds compactly in $B$, which in turns embeds continuously in $Y$,
\end{itemize}
then $(u_n)_n$ admits a strongly converging subsequence in $\Ld^p(I;B)$, provided $p<\infty$ or $r>1$. This strong convergence allows then to pass to the limit in the approximation procedure or the asymptotic limit. 

\vspace{2mm}

 The main purpose of this work is to revisit this classical result in order to handle the case of estimates arising from two particular cases : degenerate parabolic equations and incompressible Navier-Stokes equations, the latter being considered in a non-cylindrical domain. These two types of situations do not allow to apply the usual Aubin-Lions directly (we will explain why in the sequel).

\vspace{2mm}

Of course, these equations have already been  well studied in the literature as well as the difficulties arising from their nonlinearities. Hence, the novelty of this work does not concern so much the results for their own sake  (except one or two improvements) but the strategies of proof which, as far as we know, are new. In this way, we hope to give simple arguments that could possibly be reused in different contexts.

\vspace{2mm} 

Before presenting our results, let us describe the existing literature. The naming ``Aubin-Lions'' may be traced back to the seminal papers of Aubin \cite{aub} and Lions \cite{lions}, in the '$60$. However, at the same period, Dubinski\u{\i} proved a general compactness result (see \cite{dub} and the corrected version \cite{barret}) for vector-valued functions which is actually the first nonlinear counterpart of the Aubin's result (the vector space $X$ is replaced by a cone). This is why some authors refers sometimes to the Aubin-Lions-Dubinksi\u{\i} Lemma (see for instance \cite{chen}). Let us mention also the result of Kruzhkov in \cite{kru} which, though far less general than the previouses, present a different approach that we will use in section 4. Simon extended the result of Aubin and Lions to non-reflexive Banach spaces in his highly-cited paper \cite{simon}. In this paper, the condition on the time-derivative was replaced by a more general condition on time translations. The result of Simon  was further sharpened by Amann in \cite{amann} to a refined scale of spaces (including Besov spaces for instance), and broached by Roub{\'{\i}}{\v{c}}ek in a rather general setting, see \cite{roub}. At this stage, we may distinguish three possible directions of generalization (which may overlap)
\begin{itemize}
\item[(a)] \emph{Nonlinear versions of the Aubin-Lions Lemma} \\
 This corresponds to cases in which assumption (i) above is replaced by a nonlinear condition.  For instance, Maitre considers in \cite{maitre} cases in which the space $X$  is replaced by $K(X)$ where $K$ is some compact (nonlinear) mapping $K:X\rightarrow B$. This compactness result was motivated by \cite{simon,amann} and the nonlinear compactness argument Alt and Luckhaus used in \cite{alt}. In some cases, the compactness result obtained in \cite{dub} by  Dubinski\u{\i} may be seen as a consequence of the Theorem of Maitre (see \cite{barret} for more details on that point), see also \cite{chenliu} for general results of the same flavor.
\item[(b)] \emph{Discrete-in-time setting}\\  Quite often, when dealing with approximate solutions of an evolution equation, it is not straightforward to fulfill assumption (ii), since it could happen that $(u_n)_n$  satisfy only a discrete (in time) equation. This typically happens when one replaces the operator $\partial_t$ by some finite-difference approximation. Several papers deal with this issue, (based on the time translations condition of Simon) : see \cite{and,chen,chenliu,dreher,gall}. 
\item[(c)] \emph{Non-cylindrical domain} \\
A time/space domain is called \emph{cylindrical} whenever it may be written $I\times \Omega$ where $\Omega$ is some subset of $\R^d$ and $I$ some intervall of $\R$, see Figure \ref{fig:omega} for an illustration. In the study of PDEs these types of domain are used for evolution problems with a fixed spatial domain. If one wishes to consider the case of \emph{time-dependent} or \emph{moving} spatial domain, one has to consider a family of domains $(\Omega^t)_{t\in I}$, representing the motion of the spatial domain and the corresponding non-cylindrical time/space domain
\begin{align*}
\hat{\Omega} := \bigcup_{t\in I} \{t\} \times \Omega^t.
\end{align*}  
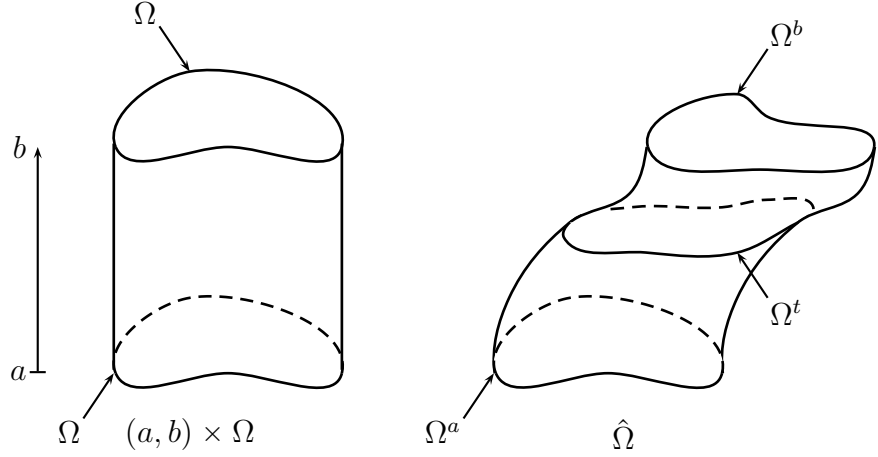
\begin{figure}[h!]
\begin{pspicture}(-2,-1)(10,5)
\rput(3,0){
\psecurve[linewidth=1pt,showpoints=false,fillstyle=solid,fillcolor=white,linestyle=dashed](1.5,0)(0.025,0)(1,1)(3,0)(1.5,0)
\psecurve[linewidth=1pt,showpoints=false,fillstyle=solid,fillcolor=white](1,1)(3,0)(1.5,0)(0.025,0)(1,1)
\rput(0,3){\psecurve[linewidth=1pt,showpoints=false,fillstyle=solid,fillcolor=white](1.5,0)(0.025,0)(1,1)(3,0)(1.5,0)
\psecurve[linewidth=1pt,showpoints=false,fillstyle=solid,fillcolor=white](1,1)(3,0)(1.5,0)(0.025,0)(1,1)}
\psline[linewidth=1pt](0,0.045)(0,3.045)
\psline[linewidth=1pt](3,0.045)(3,3.045)
\psline{->}(0.6,4.6)(1,4)
\uput{0.5mm}[ul](0.6,4.6){$\Omega$}
\psline{->}(-0.4,-0.6)(0,0)
\uput{0.5mm}[dl](-0.4,-0.6){$\Omega$}
\rput(5,0){
\psecurve[linestyle=dashed,linewidth=1pt,showpoints=false,fillstyle=solid,fillcolor=white](0.92,1.8)(1,2)(1.8,2.2)(3,2.2)(3.6,2.3)(4.2,2.2)(4,2)(3,1.6)
\psecurve[linewidth=1pt,showpoints=false,fillstyle=solid,fillcolor=white](0.5,-1)(0.01,0)(0.2,1)(1,2)(1.8,2.4)(2.01,3)(2,4)
\rput(3,0){\psecurve[linewidth=1pt,showpoints=false,fillstyle=solid,fillcolor=white](0.5,-1)(0.01,0)(0.2,1)(1,2)(1.8,2.4)(2.01,3)(2,4)}
\psecurve[linewidth=1pt,showpoints=false,fillstyle=solid,fillcolor=white,linestyle=dashed](1.5,0)(0.025,0)(1,1)(3,0)(1.5,0)
\psecurve[linewidth=1pt,showpoints=false,fillstyle=solid,fillcolor=white](1,1)(3,0)(1.5,0)(0.025,0)(1,1)
\rput(2,3){\psecurve[linewidth=1pt,showpoints=false,fillstyle=solid,fillcolor=white](1.5,0)(0.025,0)(1.2,0.7)(1.6,0.4)(3,0)(1.5,0)
\psecurve[linewidth=1pt,showpoints=false,fillstyle=solid,fillcolor=white](1,1)(3,0)(1.5,-0.3)(0.025,0)(1,1)
\psline{->}(1.6,1.3)(1.2,0.7)
\uput{0.5mm}[ur](1.6,1.3){$\Omega^b$}
}
\psecurve[linewidth=1pt,showpoints=false,fillstyle=solid,fillcolor=white](2,2.2)(1,2)(0.92,1.8)(2,1.6)(3.2,1.6)(4,2)(4.4,2.2)
\psline{->}(3.6,1)(3.2,1.6)
\uput{0.5mm}[dr](3.6,1){$\Omega^t$}
\psline{->}(-0.4,-0.6)(0,0)
\uput{0.5mm}[dl](-0.4,-0.6){$\Omega^a$}
}

\psline{|->}(-1,0)(-1,3)
\uput{1.5mm}[l](-1,0){$a$}
\uput{1.5mm}[l](-1,3){$b$}
\uput{1.5mm}[l](2,-0.8){$(a,b)\times\Omega$}
\uput{1.5mm}[l](7,-0.8){$\hat{\Omega}$}
}

\end{pspicture}
\label{fig:omega}\caption{A cylindrical domain (left) and a non-cylindrical one (right).}
\end{figure}

The study of PDEs in non-cylindrical domains leads to the following difficulty: functions $u:(t,x)\mapsto u(t,x)$ defined on $\hat{\Omega}$ \emph{may not anymore} be seen as  functions of the time-variable $t$ with value in a fixed space of functions of the $x$ variable. Typically, this forbids the assumptions (i), (ii) above and even the conclusion of compactness in $\textnormal{L}^p(I ; \B)$: the very statement of the Aubin-Lions lemma is already problematic. As far as we know, the first proof of a compactness Lemma ``à la Aubin-Lions'' in the case of a non-cylindrical domain appeared in a paper by Fujita and Sauer \cite{fuj}, for the treatment of the incompressible Navier-Stokes equations in moving domain. The method of proof (which was reused in the framework of fluid/structure interaction, see Conca \emph{et. al.} \cite{tucsnak} for instance) uses the idea that, under appropriate regularity conditions, the non-cylindrical domain $\hat{\Omega}$ is close to a finite union of cylindrical domains on which one could then use the usual Aubin-Lions Lemma.
\end{itemize}

The purpose of this study is to prove two generalizations of the Aubin-Lions Lemma. The first one corresponds to cases (a) and (b) above, and the second one to case (c). Both will be proven using totally different methods than the one developped in the above literature. Before stating precisely these two results, we introduce a few notations.

\subsection{Notations}\label{subsec:not}
The norm of a vector space  $X$ will always be denoted $\|\cdot\|_X$, with an exception for the $\textnormal{L}^p$ spaces for which we will often use the notation $\|\cdot\|_p$ if there is no ambiguity. Vectors and vector fields are written in boldface. We omit the exponent for the functionnal spaces constituted of vector fields: we denote for instance  $\textnormal{L}^2(O)$ instead of $\textnormal{L}^2(O)^d$ the set of all vector fields ${O \rightarrow\R^d}$ whose norm is square-integrable.

\vspace{2mm}

When $O$ is some open set of $\R^d$ (or $\R\times\R^d$ ...), we adopt the usual notations for the Sobolev spaces $\W^{1,p}(O)$ and $\H^m(O)$, for $p\in[1,\infty]$ and $m\in\N$. $\mathscr{D}(O)$ denotes the space test functions: smooth functions having a compact support in $O$, while $\mathscr{D}(\overline{O})$ is the restriction to $\overline{O}$ (closure of $O$) of elements in $\mathscr{D}(\R^d)$. $\H^{-m}(O)$  denotes the dual space of $\H^m_{{\footnotesize 0}}(O)$ the latter being the closure of $\mathscr{D}(O)$ under the $\H^m(O)$ norm.  If $O$ has a Lipschitz boundary and $p< d$, we denote by $p^\star$ the exponent of the Sobolev embedding $\W^{1,p}(O) \hookrightarrow \textnormal{L}^{p^\star}(O)$, that is $p^\star := dp/(d-p)$. We adopt the convention $p^\star = \infty$ when $p=d$ (the previous embedding fails in this case). Conjugate coefficient of $p$ is denoted by $p'$. We denote by $\mathscr{M}(O)$ (resp.  $\mathscr{M}(\overline{O})$) the set of finite Radon measures on $O$ (resp. $\overline{O}$) and by $\textnormal{BV}(O)$ the subset of $\Ld^1(O)$ constituted of functions having all their weak derivatives in $\mathscr{M}(O)$. If $I$ is  an intervall and $X$ some Banach (or Fréchet) vector space, we denote by $\Ld^p(I;X)$ the set of all measurable $\Ld^p$  functions from $I$ to $X$ and by $\mathscr{C}^0(I;X)$ the space of continuous functions from $I$ to $X$. When $I$ is closed, $\mathscr{M}(I;\H^{-m}(\R^d))$ is simply the dual space of $\mathscr{C}^0(I;\H^m(\R^d))$.

\vspace{2mm}

We denote by $\Ddiv(O)$ (and similarly $\Ddiv(\overline{O})$) the set of divergence-free test function with support in $O$ and by $\Ldiv^2(O)$ the subspace $\textnormal{L}^2(O)$ vector fields having a vanishing (weak) divergence.  

\vspace{2mm}

If $O\subset\R^d$ has a Lipschitz boundary we can equip $\H^{1/2}(\partial O)$ with the norm
\begin{align*}
 \| g\|_{\H^{1/2}(\partial O)}  = \inf_{v\in \H^1(O), \gamma v=g} \|v\|_{\H^1(O)},
\end{align*}
where $\gamma$ is trace operator on $\H^1(O)$. We denote by $\H^{-1/2}(\partial \Omega)$ the topological dual of $\H^{1/2}(\partial O)$. We recall that in that case (see \cite{gira} for instance), there exists a normal trace operator, that we denote by $\gamma_n$, extending the operator  $\mathscr{C}^0(\overline{O}) \ni \vv \mapsto \vv\cdot \nn \in\mathscr{C}^0(\partial O)$, where $\nn$ is the outward unit normal defined on the boundary $\partial O$, into a linear and onto map $\gamma_n:\Ldiv^2(O)\rightarrow\H^{-1/2}(\partial O)$ satisfying 
\begin{align}
\label{ineq:trace:div}\|\gamma_n \uu\|_{\H^{-1/2}(\partial \Omega)} \leq \|\uu\|_{\textnormal{L}^2(\Omega)}.
\end{align}

$\Ldiv^2(O)$ is a closed subspace of $\textnormal{L}^2(O)$: it is in fact the closure of $\Ddiv(\overline{O})$. It is however important to recall that  $\Ddiv(O)$  is not dense in $\Ldiv^2(O)$, its closure is the subspace $ \Ker \,\gamma_n$, that we will denote $\Ldivoo^2(O)$ in the sequel.

\vspace{2mm}

We adapt the previous notations for solenoidal vector fields to the case of functions depending on both time and space, that is when $O$ is an open set of $\R\times\R^d$ (first component $t$, last components $\xx$). In that case, when there is no ambiguity on the time variable, we perform a slight abuse of notation and denote for instance $\Ddiv(\R\times\R^d)$ the  set of all $\fffi\in\D(\R\times \R^d)$ such as for, for all $t\in\R$, $\fffi(t):\xx\mapsto \fffi(t,\xx) \in\Ddiv(\R^d)$.  $\Ddiv(O)$ is then just the subspace of $\Ddiv(\R\times\R^d)$ having a compact support in $O$, while $\Ldiv^2(O)$ and $\Ldivoo^2(O)$ are respectively the closure  of $\Ddiv(\overline{O})$ and $\Ddiv(O)$ in $\Ld^2(O)$. Notice that in this way we recover the definition we had without the time variable. 

\vspace{2mm}

In all this study the symbols $\dot{\in}$ and $\ddot{\in}$ will respectively mean  ``is bounded in'' and  ``is relatively compact in''. 

\vspace{2mm}

If $A$ is a connected open set of $\R^d$, and $\ep\geq 0$ we define \emph{$\ep$-interior} of $A$ as $A_\ep := \{ x \in A\,:\,\dd(x,A^c) > \ep\}$, while for  $A_{-\ep}$ denote the \emph{$\ep$-exterior} of $A$, that is   $A_{-\ep} := A + \B(0,\ep).$ One checks easily that $({A_{\ep_1})}_{\ep_2}=A_{\ep_1+\ep_2}$ and $A_0=A$. 

\vspace{2mm}

For $\sigma\in\R$ and $\hh\in\R^d$, we denote by $\lambda_\sigma$ and $\tau_{\hh}$ the shift operators in time and space respectively : if $f$ is some function depending on $(t,\xx)$, then $\lambda_\sigma f(t,x) = f(t-\sigma,\xx)$ and $\tau_{\hh} f(t,\xx) = f(t,\xx-\hh)$. 

\subsection{Main results}\label{subsec:main}

The two main results of this paper both generalize the usual Aubin-Lions Lemma in cases in which it may not be applied.

\vspace{2mm}

 The first one deals with the case when compactness on the space variable is not known on the sequence itself, but on some function of it.

\begin{theorem}\label{thm:deg}
  Consider $I\subset\R$ a non-empty closed and bounded interval, and $\Omega\subset\R^d$ a bounded open set with  Lipschitz boundary. Consider also a function $\Phi\in\mathscr{C}^1(\R,\R)$ such as $\{z \in \R\,:\, \Phi'(z)=0\}$ is finite, with $|\Phi'|$ lower bounded by a positive value near $\pm \infty$. If a sequence of $\Wl^{1,1}(I\times\Omega)$ functions satisfies $(a_n)_n\din\textnormal{L}^2(I\times\Omega)$, $(\partial_t a_n)_n \din \mathscr{M}(I;\H^{-m}(\Omega))$ and $(\nabla_{\xx} \Phi(a_n))_n \din\textnormal{L}^2(I\times\Omega)$  then $(a_n)_n\ddin\textnormal{L}^2(I\times\Omega)$.
\end{theorem} 
\begin{rmk}
This result comes under cases (a) and (b) following the description of subsection \ref{subsec:beyond} : the compactness assumption for the space variable is not known for the sequence $(a_n)_n$, but for some nonlinear function of it ; also note that the assumption on the time derivative allows Dirac masses in time. In particular, this Theorem applies for step functions in time with values in $\H^{-m}(\Omega)$.
\end{rmk}
Let us first explain why there is a real loss of information with respect to  the compactness in the space variable $x$ for $(a_n)_n$ in comparison with the usual Aubin-Lions Lemma. On each points where $\Phi'$ vanishes we may write
\begin{align*}
\nabla_{\xx} a_n = \frac{1}{\Phi'(a_n)} \nabla_{\xx} \Phi(a_n),
\end{align*}
whenever $a_n$ does not meet the set of critical points and expect an estimate for $(\nabla_{\xx} a_n)_n$, but when $a_n$ approaches a critical point, the estimate degenerates : the usual Aubin-Lions Lemma may not be invoked since no estimates for the gradient of $(a_n)_n$ may be obtained generally. 

\vspace{2mm}

The assumptions of Theorem \ref{thm:deg} are directly linked with the estimates of the equation of porous medium $\partial_t u-\Delta_{\xx} u^m =0$ (in the case of fast diffusion $m>1$)  and more generally may be useful to parabolic degenerate equations of the following form: 
\begin{align}
\label{eq:dege}\partial_t u -\div_{\xx} [ A \nabla_{\xx} \Phi(u) ] = 0,
\end{align}
where $A(t,\xx)$ satisfies some uniform coercivity condition. Indeed if, for instance, $\text{Spec}(A+{}^{t}\!A)$ is lower bounded by $\lambda>0$ uniformly in $(t,\xx)$, then, if $\Psi$ is a anti-derivative of $\Phi$, one gets easily
\begin{align*}
\int_{\Omega} \Psi(u)(s)  + \int_0^s \int_{\Omega} {}^{t}\nabla_{\xx} \Phi(\uu) A  \nabla_{\xx} \Phi(u) = \int_{\Omega} \Psi(u)(0),
\end{align*} 
whence 
\begin{align*}
\int_{\Omega} \Psi(u)(s)  +  \lambda\int_0^s \int_{\Omega} |\nabla_{\xx} \Phi(u)|^2 \leq \int_{\Omega} \Psi(u)(0),
\end{align*} 
which, under appropriate growth condition on $\Phi$ will lead to the assumptions used in Theorem \ref{thm:deg}. Typically, for the porous medium case $\Phi(x) = x^m$ with $m>1$, the previous estimate gives directly the $\textnormal{L}^2$ estimate on the gradient, and that  $u$ belongs to $\textnormal{L}^\infty(I;\textnormal{L}^{m+1}(\Omega))$ whence both the $\textnormal{L}^2$ estimate for $u$ and the $\textnormal{L}^1(I;\H^{-m}(\Omega))$ estimate for its time-derivative. 
\vspace{2mm} 

To prove Theorem \ref{thm:deg}, we will first give a general criteria to pass to the limit in a product $(a_n b_n)_n$ under assumption of weak convergence for both $a_n$ and $b_n$. This criteria seems reminiscent of the celebrated compensated compactness phenomenon exhibited by Murat and Tartar in \cite{murat,tartar} (see also \cite{ger}). However, it is not (as far as we know) a consequence of the compensated compactness theory, but share this common feature : nonlinearities are handled without insuring strong convergence for one of the sequences a compactness.  As far as we know, the strategy used in the literature to treat degenerate parabolic equations like \eqref{eq:dege}  is different of the one we followed, and often relies on the equation's structure, see \cite{vaz} and \cite{dia} for instance. Our proof is also different of the quite general approach proposed by Maitre in \cite{maitre}. A benefit of our method is that it directly applies to both step-functions (in time) and continuous functions. The result may hence be used to prove the strong compactness of a sequence defined by a semi-implicit scheme :
\begin{align*}
\frac{1}{\delta}(u_{n+1}-u_n) - \textnormal{div}_{\xx}[A \nabla_{\xx} \Phi(u_{n+1})] =0.
\end{align*}
Indeed, if $\delta=1/N$ and $\inf I = t_0<\cdots<t_N=\sup I$ is a regular discretization of the interval $I$, one defines 
\begin{align*}
\tilde{u}_N(t,x) := \sum_{k=0}^{N-1} u_k(x) \mathds{1}_{(t_k,t_{k+1})}(t),
\end{align*} 
and applies then Theorem \ref{thm:deg} to obtain the strong compactness of $(\tilde{u}_N)_N$. In fact Theorem \ref{thm:deg} is already used by the author and some collaborators in the proof of global weak solutions for a reaction/cross-diffusion system, approximated by a similar scheme, see \cite{cross}.

\vspace{2mm}

The second main result of this paper comes under case (c) following the description of subsection \ref{subsec:beyond}. We consider a family $(\Omega^t)_{t\in  [a,b]}$ given by the motion of a Lipschitz, connected and bounded reference domain $\Omega\subset\R^d$: $\forall t\in [a,b]$, $\Omega^t:= \mathcal{A}_t(\Omega)$, where for all $t$, $\mathcal{A}_t:\R^d\rightarrow\R^d$ is a  $\mathscr{C}^1$-diffeomorphism. The regularity of the motion is described through the function $\Theta:(t,\xx)\mapsto \mathcal{A}_t(\xx)$, for which we assume the following
\begin{assumption}\label{ass:A1}
The function $\Theta$ belongs to $\mathscr{C}^0([a,b];\mathscr{C}^1(\R^d))$.
\end{assumption}
We will need this Assumption for every result of section \ref{sec:mov} dealing with the family $(\Omega^t)_{t\in I}$ (including Theorem \ref{thm:aubin} below) exception made for Theorem \ref{thm:aubin:gen} (see Remark \ref{rmk:ass}).

\vspace{2mm}

We work on the non-cylindrical domain 
\begin{align*}
\Omegah := \bigcup_{a < t< b} \{t\}\times\Omega^t,
\end{align*}
for which we have the following compactness result 
\begin{theorem}\label{thm:aubin} Consider a sequence $(\uu_n)_n\in\Ldiv^2(\R\times\R^d)$ vanishing outside $\Omegah$.  Assume that  $(\uu_n)_n, (\nabla_\xx \uu_n)_n \dot{\,\in\,} \textnormal{L}^2(\Omegah)$, and $(\uu_n)_n\din\textnormal{L}^\infty(\R;\textnormal{L}^2(\R^d))$. Eventually, assume the existence of a constant $\Cc>0$ and an integer $N>0$ such as for all divergence-free test function $\ppsi \in\Dd(\hat{\Omega})$,
\begin{align}
\label{ineq:lem}\left|\langle  \partial_{t} \uu_n, \ppsi\rangle \right| \leq \Cc \sum_{|\alpha|\leq N} \|\partial_{\xx}^\alpha\ppsi\|_{\textnormal{L}^2(\hat{\Omega})}.
\end{align}
Then $(\uu_n)_n\ddot{\,\in\,}\textnormal{L}^2(\hat{\Omega})$.
\end{theorem}
The bounds assumed on the sequence $(\uu_n)_n$ are typical of the incompressible Navier-Stokes equations, considered in a non-cylindrical domain. Indeed, recall the incompressible Navier-Stokes equations :
\begin{align}
 \label{eq:ns}\partial_t \uu + \uu\cdot \nabla_{\xx} \uu -\Delta_{\xx} \uu + \nabla_{\xx} p &= 0,\\
 \label{eq:nsdiv}\div_{\xx} \uu = 0.
\end{align}
If these equations are considered on $\Omegah$ with appropriate boundary conditions, one has the following formal energy equality
\begin{align*}
\frac{\dd}{\dd t}\int_{\Omega_t} |\uu(t,\xx)|^2 \dd \xx + \int_0^t
\int_{\Omega_s} |\nabla_{\xx}\uu(s,x)|^2 \dd \xx\dd s = 0,
\end{align*}
which explain the assumptions on $(\uu_n)_n$ and $(\nabla_{\xx}\uu_n)_n$. But what about $(\partial_t \uu_n)_n$ ? In fact it is not possible to estimate directly $\partial_t \uu$ from the equation (as this is usually the case in the Aubin-Lions Lemma), because of the pressure term $\nabla_{\xx} p$, from which usually very less is known. However, testing \eqref{eq:ns} against a smooth divergence free vector field $\ppsi$ gives rise to the estimate \eqref{ineq:lem}. This estimate is the equivalent of estimate (ii) in subsection \ref{subsec:beyond}. Notice, that we may not write it as $(\partial_t \uu_n)_n\din \Ld^r(I;Y)$ precisely because the domain is non-cylindrical, whence the dual formulation \eqref{ineq:lem}. In fact Theorem \ref{thm:aubin}  has been partially tailored for the study of a fluid/kinetic coupling in a moving domain, and will soon be used in this context by the author and some collaborators. Of course, normally \eqref{eq:ns} -- \eqref{eq:nsdiv} are completed with boundary (and initial) conditions, and the assumptions of Theorem \ref{thm:aubin} suggest that we may only handle homogeneous Dirichlet boundary conditions (that is, $\uu$ vanishes on the boundary). More general boundary conditions may in fact be handled, this is the purpose of Corollary \ref{coro:thm:aubin}.

\vspace{2mm}

As explained in case (c) of subsection \ref{subsec:beyond}, the existing proofs of such a compactness result in  non-cylindrical domain are based on the following observation: $\hat{\Omega}$ may be decomposed, up to some small subset, into a finite union of cylindrical domains. On each one of these, one may then invoke the usual Aubin-Lions Lemma. We give here a totally different proof which,  gives strong compactness, \emph{without using} the standard Aubin-Lions lemma. In particular, the method applies in the case of cylindrical domains and gives hence a new proof of the Aubin-Lions lemma (but not in a framework as general as in \cite{simon}). It avoids also the ``slicing'' step for the non-cylindrical domain, which leads in \cite{fuj} to intricate assumptions for the regularity of the motion of the domain whereas Assumption \ref{ass:A1} is simpler and weaker. The method we elaborate is far simpler when one replaces \eqref{ineq:lem} by dual estimate against all test function (and not only divergence-free), see Theorem \ref{thm:aubin:gen}. In fact, the main difficulty of Theorem \ref{thm:aubin} concerns this point.

\subsection{Structure of the paper} 

Let us now describe the structure of this paper. Section \ref{sec:compe} is devoted to results reminiscent of the celebrated compensated compactness phenomenon exhibited by Murat and Tartar in \cite{murat,tartar} (see also \cite{ger}). The results of section \ref{sec:compe} will be used in the proof of Theorem \ref{thm:deg} but are also interesting for their own sake. As a byproduct, we explain for instance how to handle one of the nonlinearities arising in \cite{jabin} in the study of a hydrodynamic limit, in dimension $2$. Section \ref{sec:degen} focuses on the proof of Theorem \ref{thm:deg}. Finally, in section \ref{sec:mov} we prove Theorem \ref{thm:aubin}. Subsection \ref{subsec:epint} gives general uniform properties for $\ep$-interior sets, as defined in subsection \ref{subsec:not}. The proof of Theorem \ref{thm:aubin} takes it simpler form in the case when the dual estimate on the time derivative \eqref{ineq:lem} is known for all test functions, we hence dedicate subsection \ref{subsec:easy} to this simplified framework. In the general case the proof is a bit more involved and relies on properties of divergence-free vector fields that we describe in subsection \ref{subsec:div}. The core of the proof of Theorem \ref{thm:aubin} is contained in subsection \ref{subsec:nav}. Section \ref{sec:mov} is totally independent of sections \ref{sec:compe} and \ref{sec:degen}. Finally we prove in appendix section \ref{sec:appe} two technical results.

\section{Weak convergence of a product}\label{sec:compe}


Proposition \ref{propo:compe:space} below is directly inspired from an argument used in \cite{diperna} (in the $\textnormal{L}^\infty / \textnormal{L}^1$ framework). It was already stated and proved in \cite{saad} (in a periodic setting), but we reproduce it here with the proof, for the sake of completeness and give also two other variants. Let us first treat the case without boundary, in order to use freely the convolution in the $\xx$ variable. Since this work is motivated by evolution equations, $t$ represents here the time variable and $I$ is some compact intervall of $\R$, but it is quite clear that similar results may be obtained replacing $I$ by some bounded open set of $\R^d$, with suitable assumptions.

\vspace{2mm}

We will use repeatedly a sequence $(\ffi_k)_k$  of nonnegative even mollifiers (in space only) : $\ffi_k(\xx) := k^{d}\ffi( k \xx)$, where $\ffi$ is some smooth even nonnegative function with support in the unit ball of $\R^d$. In all this section the convolution $\star$ has to be understood in the space variable $\xx$ only.

\vspace{2mm}

We start with a ``commutator Lemma'' reminiscent of the usual Friedrichs Lemma (both being key elements of \cite{diperna}), the only difference is that no differential operation is involved here, but the convergence holds uniformly in $n$.

\begin{lem}\label{ll}
Let $q\in[1,\infty]$ and $p\in[1,d]$ and $I\subset\R$ a non-empty closed and bounded interval. Consider  $(a_n)_n \din \textnormal{L}^q(I;\W^{1,p}(\R^d))$ and $(b_n)_n \din \textnormal{L}^{q'}( I ;\textnormal{L}^{\alpha'}(\R^d))$ with $\alpha < p^\star$. Then the commutator (convolution in $\xx$ only) 
\begin{align*}
S_{n,k} &:= a_n\, (b_n\star\ffi_k)-(a_n\, b_n)\star\ffi_k
\end{align*}
goes to $0$  in $\textnormal{L}^1(\R\times\R^d)$ as $k\rightarrow +\infty$, uniformly in $n$.
\end{lem}
\begin{proof}
\medskip
Since $(a_n)_n \din\textnormal{L}^q(I;\W^{1,p}(\R^d))$ and $\alpha<p^\star$, we have
\begin{align*}
(\tau_{\hh} a_n-a_n)_n  \conv{\hh}{0} 0 \text{ in } \textnormal{L}^q(I;\textnormal{L}^\alpha(\R^d)),
\end{align*}
uniformly in $n$. We now follow \cite{diperna} and write the following equality for the commutator
  \begin{align}\label{commut}
S_{n,k}(t,\xx) =\int_{|\yy|\leq 1/k}\Big[a_n(t,\xx)-a_n(t,\xx-\yy)\Big]\,
b_n(t,\xx-\yy)\, \ffi_k(\yy)\, \dd \yy,
\end{align}
whence thanks to Fubini's Theorem, integrating on $I\times\R^d$
\begin{align*}
\|S_{n,k}\|_{\textnormal{L}^1(I\times\R^d)}\leq \|b_n\|_{\textnormal{L}^{q'}(I;\textnormal{L}^{\alpha'}(\R^d))}  \int_{|\yy|\leq 1/k} |\ffi_k(\yy)| \,
\|\tau_\yy a_n-a_n\|_{\textnormal{L}^{q}(I;\textnormal{L}^{\alpha}(\R^d))}
\, \dd \yy,
\end{align*}
which yields the desired uniform convergence.$\qedhere$ \end{proof}

\begin{proposition}\label{propo:compe:space}
Let $q\in[1,\infty]$ and $p\in[1,d]$ and $I\subset\R$ a non-empty segment. Consider  $(a_n)_n \din \textnormal{L}^q(I;\W^{1,p}(\R^d))$ and $(b_n)_n \din \textnormal{L}^{q'}( I ;\textnormal{L}^{\alpha'}(\R^d))$ respectively weakly or weakly$-\star$ converging in these spaces to $a$ and $b$. Assume that $\alpha < p^\star$. If $(\partial_t b_n)_n\din \mathscr{M}(I;\H^{-m}(\R^d))$ for some $m\in\N$ then, up to a subsequence, we have the following vague convergence in $\mathscr{M}(I\times\R^d)$ (\emph{i.e.} with $\mathscr{C}_c^0(I\times\R^d)$ test functions) : 
\begin{align}
\label{conv:compe}(a_n \,b_n)_n \convw{n}{+\infty} a\,b.
\end{align} 
\end{proposition}
\begin{rmk} 
As explained in the introduction, we may recognize in this lemma a kind of compensated compactness flavor since, in the above result, neither $(a_n)_n$ nor $(b_n)_n$ do converge strongly: both may oscillate but only in a somehow compatible way. Nevertheless, as far as we know, this result does not exactly recast in the work of Murat and Tartar. 
\end{rmk}
\begin{rmk}
If $a_n=b_n$, one gets strong compactness for $(a_n)_n$. Of course this situation is nothing else than a particular case of the usual Aubin-Lions Lemma.
\end{rmk}
\begin{proof}
Fix $N\in\N$ and denote $B_N:=\B(0,N)\subset\R^d$ (open ball of radius $N$), $O_N:=I\times B_N$.  By a standard diagonal argument, it suffices to prove, up to a subsequence, that $(a_n b_n)_n \rightharpoonup ab$ in the vague topology of $\mathscr{M}(O_N)$, that is, against $\mathscr{C}_c^0(O_N)$ test functions.  

\vspace{2mm} 

Let us follow the following routine to conclude the proof.

\begin{itemize}
\item[\emph{Step 1.}] We have clearly
\begin{align*}
a \,(b\star\ffi_k) \operatorname*{\longrightarrow}_{k \rightarrow +\infty} a\, b,\text{ in }\textnormal{L}^1(O_N) \text{ strong } .
\end{align*}
\item[\emph{Step 2.}] Since $(\partial_t b_n)_n \din\mathscr{M}(\R;\Hh^{-m}(\R^d))$,  we get easily $(b_n\star \ffi_k)_n \din\textnormal{BV}(O_N)$ so that, we can choose (but we don't write it explicitly) a common (diagonal) extraction such as, for all fixed $k$, $(b_n \star\ffi_k)_n$ converges a.e. on $O_N$  to $b\star\ffi_k$. We hence deduce from the preceding fact that (for all fixed $k$)
\begin{align*}
(a_n \,(b_n\star\ffi_k))_n \operatorname*{\longrightharpoonup}_{n \rightarrow +\infty} 
a\, (b\star\ffi_k) \text{ in } \textnormal{L}^1(O_N)\text{ weak.}
\end{align*}
Indeed, we have (at least) the weak convergence of $(a_n)_n$ towards $a$ in $\Ld^1(O_N)$, and since $(b_n\star\ffi_k)_n$ converges to $b\star\ffi_k$ a.e. on $O_N$, the above convergence follows from the estimate $(b_n\star\ffi_k)_n \din\Ld^\infty(O_N)$, the latter being a direct consequence of $(\partial_t b_n)_n \din\mathscr{M}(I;\Hh^{-m}(\R^d))$.
\item[\emph{Step 3.}] 
  From Lemma \ref{ll} we infer
\begin{align*}
\sup_n \|a_n(b_n\star\ffi_k)-(a_n b_n)\star\ffi_k\|_1 \conv{k}{+\infty} 0.
\end{align*}
\item[\emph{Step 4.}] For a fixed $\theta \in \mathscr{C}^0_c(O_N)$ 
\begin{align*}
\langle (a_n \,b_n)\star\ffi_k-a_n\, b_n,\theta\rangle \operatorname*{\longrightarrow}_{k \rightarrow +\infty} 0,
\end{align*}
uniformly in $n$. Indeed, since $\ffi_k$ is even, we may write
\begin{align*}
\langle (a_n \,b_n)\star\ffi_k-a_n\, b_n,\theta\rangle = \langle a_n\, b_n,\theta\star\ffi_k-\theta\rangle,
\end{align*}
and the right-hand side tends to $0$ with the desired uniformity because $(a_n \,b_n)_n$ is bounded in $\textnormal{L}^1(I\times\R^d)$, and $(\theta\star\ffi_k-\theta)_k$ goes to $0$ in $\textnormal{L}^\infty(I\times\R^d)$ ($\theta$ is uniformly continuous ).
\item[\emph{Step 5.}] Write
\begin{align*}
a\,b-a_n\, b_n &= a\,b- a\,(b\star\ffi_k)\\
&+ a\,(b\star\ffi_k)-a_n\, (b_n\star\ffi_k)\\
&+ a_n\,(b_n\star\ffi_k)-(a_n\, b_n)\star\ffi_k\\
&+ (a_n \,b_n)\star\ffi_k -a_n\, b_n.
\end{align*}
Fix $\theta\in\mathscr{C}^0_{c}(O_N)$, multiply the previous equality by $\theta$ and integrate over $O_N$.  In the right-hand side, line number $i\in\{1,2,3,4\}$ corresponds to the Step $i$ proven previously. We choose first $k$ to handle (uniformly in $n$) all the lines of the right-hand side, except the second one. Then, we
 choose the appropriate $n$ to handle the second line, thanks to Step 2.
 This concludes the proof of Proposition \ref{propo:compe:space}.$\qedhere$
\end{itemize}
\end{proof}

It is worth noticing that in the proof of Proposition \ref{propo:compe:space}, the only step in which the assumption $(\nabla_{\xx} a_n)\din\textnormal{L}^q(I;\textnormal{L}^p(\R^d))$ is crucial is Step 3, for the treatment of the ``commutator''.  In fact one can easily relax this assumption in the following way. 
If $X$ denotes some abstract function space of the $\xx$ variable,  a sufficient condition on $(a_n)_n$  (to handle Step 3) is that $\|\tau_{\hh} a_n - a_n\|_{\textnormal{L}^q(I;X)} $ goes to $0$ with $\hh$, provided that $(b_n)_n \din \textnormal{L}^{q'}(I;X')$. Bearing this in mind, one may for instance prove the following Proposition

\begin{proposition}\label{propo:compe:space2}
Consider a non-empty segment $I\subset\R$ and two sequences $(a_n)_n$ and $(b_n)_n$ weakly converging in $\textnormal{L}^1(I\times\R^2)$ to $a$ and $b$. Assume that $(a_n)_n\din\textnormal{L}^q(I;\H^1(\R^2))$ and that $(|b_n|\log |b_n|)_n \din\textnormal{L}^{q'}(I;\textnormal{L}^1(\R^2))$ and $(\partial_t b_n)_n \din\mathscr{M}(I;\H^{-m}(\R^2))$ for some $m\in\N$. Then, up to a subsequence, we have the following vague convergence in $\mathscr{M}(I\times\R^d)$ (\emph{i.e.} with $\mathscr{C}_c^0(I\times\R^d)$ test functions) : $\ds(a_n \,b_n)_n \convw{n}{+\infty} a\,b$.
\end{proposition} 
\begin{rmk}
Since the cornerstone in the below proof is the Moser-Trudinger inequality, this Lemma may of course be generalized to $\R^d$, replacing $\H^1(\R^2)$ by $\W^{1,d}(\R^d)$.
\end{rmk}
\begin{proof}
Let us sketch the proof briefly. As before we work on $B_N$ for the $\xx$ variable and without more precision $\|\cdot\|_p$ will denote the $\textnormal{L}^p(B_N)$ norm. Consider the two convex functions $\Phi:\R_+\ni x\mapsto e^x-x-1$ and $\Psi:\R_+\ni x\mapsto(1+y)\log(1+y)-y$. For any measurable function $f:\R^2\rightarrow\R$ such as $\Phi(|f|)\in\textnormal{L}^1(B_N)$ we recall the Luxemburg gauge
\begin{align*}
 \|f\|_{\Phi} := \inf\big\{ a>0\,:\, \|\Phi(|f|/a)\|_1 \leq 1 \big\},
\end{align*}
and define in a similar way $\|\cdot\|_{\Psi}$. It is straightforward to check that $\Phi$ and $\Psi$ are convex-conjugate of one another, and satisfy the Young inequality $xy \leq \Phi(x) + \Psi(y)$. One may then deduce the following generalized Hölder inequality, that is for all measurable functions $f$ and $g$ such as $\Phi(|f|) \in\textnormal{L}^1(\B_N)$ and $\Psi(|g|) \in\textnormal{L}^1(\B_N)$,
\begin{align}
\label{ineq:holder} \| fg\|_1 \leq \|f\|_{\Phi}\|g\|_{\Psi}.
\end{align} 
For more details on the previous inequality, Luxemburg gauge and Orlicz spaces, see \cite{adams} for instance. 
\vspace{2mm}
Since $(a_n)_n\din\textnormal{L}^q(I;\H^1(\R^2))$, we deduce from the Moser-Trudinger inequality (see again \cite{adams}) that $(e^{\alpha a_n^2})_n \din\textnormal{L}^q(I;\Ll^1(\R^2))$, for a positive constant $\alpha$ small enough (in fact $\alpha < 4\pi$). In particular, using \eqref{ineq:holder} and the bound $(|b_n|\log |b_n|)_n \din\textnormal{L}^{q'}(I;\textnormal{L}^1(\R^2))$, one gets $(a_n b_n)_n \din\textnormal{L}^1(I\times B_N)$. This estimate is sufficient to reproduce all Steps of the proof of Proposition \ref{propo:compe:space}, except the third one, and as mentionned before, only Lemma \ref{ll} has to be examined. 

\vspace{2mm}

The Moser-Trudinger inequality aforementionned is based on the following fact : there exists a universal constant $C$ such as, for all $p< \infty$, and all $f\in\H^1(\R^2)$, $\|f\|_{\textnormal{L}^p(\R^d)} \leq C \sqrt{p} \|f\|_{\H^1(\R^d)}$ (see again \cite{adams} for more details). Using this, and expanding the series defining $\Phi$, one may show easily the existence of a continuous and nondecreasing function $\ffi:\R_+\rightarrow\R_+$, vanishing in $0$, such as, for any $f$ in the unit ball of $\H^1(\R^2)$, 
\begin{align}
\label{ineq:orlic}\|f-\tau_{\hh} f\|_{\Phi} \leq \ffi(|\hh|).
\end{align}


The proof of Proposition \ref{propo:compe:space2} follows then using \eqref{ineq:orlic} and inequality \eqref{ineq:holder} in the expression giving the commutator \eqref{commut}. 
$\qedhere$
\end{proof}
The bounds assumed in Proposition \ref{propo:compe:space2} are not coming from nowhere. There are for instance the one obtained for the density of particles and the fluid velocity in the hydrodynamic limit studied by Goudon, Jabin and Vasseur in \cite{jabin} (in dimension $2$). We hence recover by compactness one of the nonlinear limit (in fact the easier one) explored in \cite{jabin}. Let us mention that in \cite{jabin} the authors used a relative entropy method, whence a passage to the limit only under the assumption of preparation of the data, an assumption that we obviously do not need to apply Proposition \ref{propo:compe:space2}. An other difference is that our method is always global (in time) whereas the relative entropy method is usually limited by the existence of regular solutions for the limit system ; but of course, in dimension $2$, global regular solutions for the density-dependent incompressible Navier-Stokes are known to exist (see \cite{danchin}), so that the two approaches rejoin on that point.

\vspace{2mm}
Let us conclude this section by giving a version of Proposition \ref{propo:compe:space} in the case of a bounded domain $\Omega$ using a simple localization argument: 
\begin{proposition}\label{propo:compe}
Let $q\in[1,\infty]$ and $p\in[1,d]$, $I\subset\R$ a non-empty segment and $\Omega\subset\R^d$ a bounded open set with Lipschitz boundary. Consider two sequences $(a_n)_n \din \textnormal{L}^q(I;\W^{1,p}(\Omega))$ and $(b_n)_n \din \textnormal{L}^{q'}(I;\textnormal{L}^{\alpha'}(\Omega))$ respectively weakly or weakly$-\star$ converging in these spaces to $a$ and $b$. Assume that $\alpha < p^\star$. If $(\partial_t b_n)_n\din \mathscr{M}(I;\H^{-m}(\Omega))$ for some $m\in\N$ then, up to a subsequence, we have the following weak$-\star$ convergence in $\mathscr{M}(I\times\overline{\Omega})$  (\emph{i.e.} with $\mathscr{C}^0(I\times\overline{\Omega})$ test functions) : 
\begin{align}
\label{conv:compe}(a_n \,b_n)_n \convw{n}{+\infty} a\,b.
\end{align} 
\end{proposition}
\begin{proof}
First notice that, since $\alpha ' > (p^\star )'$, we have by Sobolev embedding $(a_n b_n)_n \din\textnormal{L}^1_t(\textnormal{L}^r_{\xx})$ for some $r>1$, whence uniform absolute continuity in the $\xx$ variable, in the sense that 
\begin{align*}
\sup_{n\in\N} \int_{I\times E} |a_n b_n| \conv{\mu(E)}{0} 0,
\end{align*}
where $E$ denotes any measurable subset of $\Omega$ and $\mu$ the Lebesgue measure on $\R^d$. Now, we just pick a sequence of functions $(\theta_k)_k\in\mathscr{D}(\Omega)$ bounded by $1$, and equaling this value on a sequence of compact sets $K_k$ such as $\mu(\Omega \setminus K_k) \leq 1/k$. 

\bigskip
When $k$ is fixed, the sequences $(\theta_k a_n)_n$ and $(\theta_k b_n)_n$ (extended by $0$ outside $\Omega$) verifies all the assumptions of Proposition \ref{propo:compe:space}, whence (up to a subsequence) the expected weak convergence for the product $(\theta_k^2 a_n b_n)_n$ in $\mathscr{M}(I\times\overline{\Omega})-\star$. We eventually get \eqref{conv:compe} by writting, for any test function $\ffi\in\mathscr{C}^0(I\times\overline{\Omega})$  
\begin{align*}
\langle a_nb_n,\ffi \rangle =  \langle \theta_k^2 a_n b_n,  \ffi \rangle  + \langle a_nb_n,(1-\theta_k^2)\ffi \rangle, 
\end{align*}
since the second term of the r.h.s. is going to $0$ with $1/k$ uniformly in $n$, and one may extract diagonally allong the $k$'s to handle the first term of the r.h.s. $\qedhere$\end{proof}

\section{Proof of Theorem \ref{thm:deg}} \label{sec:degen}
As explained in the introduction, the equality
\begin{align*}
\nabla_{\xx} a_n = \frac{1}{\Phi'(a_n)} \nabla_{\xx} \Phi(a_n),
\end{align*}
may not be used to recover an estimate on $(\nabla_{\xx} a_n)_n$ using $(\nabla_{\xx} \Phi(a_n))_n\din\Ld^2(I\times\Omega)$ because such estimate degenerates when $a_n$ approaches a a critical point of $\Phi$. Replacing $(a_n)_n$ by a truncation like $\beta(a_n)$ where  $\beta$ is some smooth function erasing the critical values is of course a natural strategy. In this way, $\beta(a_n)$ will indeed have compactness in the space variable $\xx$ through a nice control of its gradient \emph{but}, all the information on the time variable will be lost : $\partial_t \beta(a_n) = \beta'(a_n) \partial_t a_n$ does not - generally - give a good bound, since we have only $(\partial_t a_n)_n \din\mathscr{M}(I;\H^{-m}(\Omega))$ and we cannot expect $\beta'(a_n) \din\mathscr{C}^0(I; \H^m(\Omega))$ ($m$ is arbitrarly large). But using Proposition \ref{propo:compe}, we may hope to pass weakly to the limit in $a_n \, \beta(a_n)$, an expression which is not far from being quadratic in $a_n$. 

\bigskip

Let us now write this in detail.

\begin{proof}[Proof of Theorem \ref{thm:deg}]
 All the coming facts and their proofs are true up to some (countable number of) extractions that we don't mention in the sequel. Denote by $(z_i)_{1\leq i\leq N}$ the set of critical points of $\Phi$. Take $\ep>0$ small enough so that the intervals $J_i^\ep:=[z_i-\ep,z_i+\ep]$ do not overlap, and denote by $J^\ep$ the reunion of these intervals. We may find a function $\beta_\ep\in\mathscr{C}^1(\R)$ such as $\beta_\ep(z) = z$  outside $J^\ep$, $\beta_\ep(z)= z_i + \Phi(z) -\Phi(z_i)$ on $J_i^{\ep/2}$, $\beta'_\ep\in\textnormal{L}^\infty(\R)$ and for a constant $C>0$ independent of $\ep$,
\begin{align}
\label{beta1}\|\beta_\ep - \Id\|_{\textnormal{L}^\infty(\R)} &\leq C \ep.
\end{align}
Now fix $\ep>0$ and write, for $a_n \notin \{z_i\,:1\leq i \leq N\}$
\begin{align*}
\nabla_{\xx} \beta_\ep(a_n) = \frac{\beta_\ep'(a_n)}{\Phi'(a_n)} \nabla_{\xx} \Phi(a_n).
\end{align*}
Since $\beta_\ep' \in\textnormal{L}^\infty(\R)$ and $|\Phi'|$ is lower bounded by a positive value outside $J^{\ep/2}$, we have  (with a bound depending on $\ep$)
\begin{align*}
(\mathds{1}_{a_n \notin J^{\ep/2}}\nabla_{\xx} \beta_\ep(a_n))_n \din\textnormal{L}^2(I\times\Omega).
\end{align*}
On the other hand, if $a_n \in J_i^{\ep/2}$, then by definition $\beta_\ep'(a_n) = \Phi'(a_n)$, so that 
\begin{align*}
\mathds{1}_{a_n \in J^{\ep/2}_i}\nabla_{\xx} \beta_\ep(a_n) &= \mathds{1}_{a_n \in J^{\ep/2}_i}\beta_\ep'(a_n) \nabla_{\xx} a_n  \\
&= \mathds{1}_{a_n \in J^{\ep/2}_i} \Phi'(a_n) \nabla_{\xx} a_n\\
& = \mathds{1}_{a_n \in J^{\ep/2}_i} \nabla_{\xx} \Phi(a_n) \din \textnormal{L}^2(I\times\Omega).
\end{align*}
At the end of the day we obtained $(\nabla_{\xx} \beta_\ep(a_n))_n \din\textnormal{L}^2(I\times\Omega)$. Using $(a_n)_n\din\textnormal{L}^2(I\times\Omega)$ and \eqref{beta1} we get $(\beta_\ep(a_n))_n \din\textnormal{L}^2(I\times\Omega)$. Let us denote by $a$ and $a^\ep$ the corresponding weak limits.  Since $(\partial_t a_n)_n\din\mathscr{M}(I;\H^{-m}(\Omega))$ we may use Proposition \ref{propo:compe} with $\mathds{1}\in\mathscr{C}^0(I\times\overline{\Omega})$ as a test function and get
\begin{align*}
 \int_{I\times\Omega} a_n \, \beta_\ep(a_n) \conv{n}{+\infty}  \int_{I\times\Omega} a_n \, a^\ep. 
\end{align*}
But because of \eqref{beta1}, we have 
\begin{align}
\label{beta:an}\| a_n -\beta_\ep(a_n) \|_{\textnormal{L}^2(I\times\Omega)} \leq C |I\times\Omega|^{1/2} \ep,
\end{align}
whence by weak lower semicontinuity
\begin{align}
\label{beta:a}\| a -a^\ep \|_{\textnormal{L}^2(I\times\Omega)} \leq C |I\times\Omega|^{1/2} \ep.
\end{align}
Now we may eventually write 
\begin{align*}
\int_{I\times\Omega} a_n^2 =  \int_{I\times\Omega} a_n \beta_\ep(a_n) + \int_{I\times\Omega} a_n (a_n-\beta_\ep(a_n)) ,
\end{align*} 
whence, using $(a_n)_n\din\textnormal{L}^2(I\times\Omega)$, the Cauchy-Schwarz inequality and \eqref{beta:an}
\begin{align*}
\operatorname*{\overline{\lim}}_{n\rightarrow\infty}\int_{I\times\Omega} a_n^2 \leq   \int_{I\times\Omega} a\, a^\ep +  C \ep,
\end{align*} 
where we changed the constant $C$ (still independent of $n$ and $\ep$). Using \eqref{beta:a} we have, changing $C$ again,
 \begin{align*}
\operatorname*{\overline{\lim}}_{n\rightarrow\infty}\int_{I\times\Omega} a_n^2 \leq   \int_{I\times\Omega} a^2 +  C \ep,
\end{align*} 
which allows to get the strong convergence of $(a_n)_n$ towards $a$.$\qedhere$
\end{proof}

\section{Proof of Theorem \ref{thm:aubin}}\label{sec:mov}

We recall (see subsection \ref{subsec:main}) the setting. We consider the open subset of $\R\times\R^d$
\begin{align*}
\hat{\Omega} := \bigcup_{a < t< b} \{t\} \times \Omega^t,
\end{align*}
where the family of connected bounded Lipschitz open sets $(\Omega^t)_{t}$ are given by the deformation of a Lipschitz, connected and bounded reference domain $\Omega\subset\R^d$: $\forall t\in (a,b)$ (closed bounded interval), $\Omega^t:= \mathcal{A}_t(\Omega)$, where for all $t$, $\mathcal{A}_t:\R^d\rightarrow\R^d$ is a $\mathscr{C}^1$-diffeomorphism. Recall also Assumption \ref{ass:A1} which supposes that $\Theta:(t,\xx)\mapsto \mathcal{A}_t(\xx)$  lie in $\mathscr{C}^0([a,b];\mathscr{C}^1(\R^d))$.

\vspace{2mm}

Under Assumption \ref{ass:A1} is satisfied, we have in particular the existence of two positive constants  $\alpha,\beta >0$  such as
\begin{align}
\label{ineq:jacob} \forall (t,\yy)\in [a,b] \times \Omega,\quad  \alpha \leq \J( \Theta)(t,\yy) \leq \beta,
\end{align}
where $\J(\Theta)$ is the Jacobian $\J(\Theta):(t,\yy)\mapsto |\det \nabla_{\yy} \mathcal{A}_t|(t,\yy)$. 

\vspace{2mm}

The previous estimate will allow us to use the change of variable $\xx=\mathcal{A}_t(\yy)$ to transport estimates from $\Omega$ to $\Omega_t$. For instance, when $p<d$, if $\Ss_\Omega$ is the Sobolev constant of the embedding $\W^{1,p}(\Omega)\hookrightarrow\textnormal{L}^{p^\star}(\Omega)$, we get the following uniform (in time) Sobolev estimate
\begin{align}
\label{ineq:sobot}\forall t\in[a,b],\quad\forall v\in\W^{1,p}(\Omega^t),\quad \| v\|_{\textnormal{L}^{p^\star}(\Omega^t)} &\leq \K_p \,\|v\|_{\W^{1,p}(\Omega^t)},
\end{align}
where $K_p := \Ss_\Omega\beta^{1/p^\star}\alpha^{-1/p}$. We obviously have a similar estimate when $p\geq d$, replacing the exponent $p^\star$ by any finite $q$ (and changing the constant). 


\subsection{Uniform properties for $\ep$-interior sets}\label{subsec:epint}
For $A$ a connected open set of $\R^d$, and $\ep\geq 0$ recall the definition of the $\ep$-interior of $A$ (see subsection \ref{subsec:not}), that we denote $A_\ep$. Since $\Omega$ is Lipschitz we have the elementary proposition
\begin{proposition}\label{propo:tube}
There exists $\gamma>0$ and $C_\gamma>0$ such as, for any $\ep\in[0,\gamma)$, the open set $\Omega_\ep$ has a Lipschitz boundary of constant at most $C_\gamma$.
\end{proposition}
Using this Proposition, we can prove the following one 
\begin{proposition}\label{propo:poinc} Consider $\gamma$ the positive number defined in Proposition \ref{propo:tube}. Then, for $\ep\in[0,\gamma)$, the open sets $\Omega_\ep$ share a common Poincaré-Wirtinger constant, that is : there exists a positive constant $\Cc_{\Omega,\gamma}$ depending only on $\Omega$ and $\gamma$, such as for all $v\in\H^1(\Omega_\ep)$ having a vanishing mean-value, we have
\begin{align*}
 \|v\|_{\textnormal{L}^2(\Omega_\ep)} \leq \Cc_{\Omega,\gamma} \|\nabla v\|_{\textnormal{L}^2(\Omega_\ep)}.
\end{align*}
\end{proposition}
\begin{rmk}
The dependence of the Poincaré-Wirtinger constant with respect to the domain of study $\Omega$ is a difficult subject which has led to numerous articles (see for instance the recent paper \cite{ruiz} and the reference therein) ; the literature thus covers quite widely our case, but we give in the appendix section \ref{sec:appe} a rather short and elementary proof in this particular case.
\end{rmk}

For time-space domains  we perform a slight abuse of notation (in contradiction with the one used for $\ep$-interior sets) :
\begin{align*}
\Omegah_\ep &:= \bigcup_{t\in I} \{t\} \times \mathcal{A}_t(\Omega_\ep),
\end{align*}
so that $\Omegah_\ep$ is defined by the motion of the $\ep$-interior set $\Omega_\ep$. We denote 
\begin{align*}
\partial\Omegah &:=  \bigcup_{a < t < b} \{t\} \times \partial\Omega^t,\\
\partial\Omegah_\ep &:= \bigcup_{a < t < b} \{t\} \times \mathcal{A}_t(\partial \Omega_\ep),
\end{align*}
which is again an abuse of notation:  $\partial\Omegah$ and $\partial\Omegah_\ep$ are not the boundaries (in $\R\times\R^d$) of respectively $\Omegah$ and $\Omegah_\ep$. 

\vspace{2mm}

Although $\mathcal{A}_t(\Omega_\ep)$ is not the $\ep$-interior  of $\Omega^t$, one has the following elementary Proposition

\begin{proposition}\label{propo:tubular} 
We have
\begin{enumerate}
\item For all $t\in[a,b]$ and all $\ep>0$, $\mathcal{A}_t(\partial\Omega_\ep)$ is the boundary in $\R^d$ of $\mathcal{A}_t(\Omega_\ep)$.
\item There exists $\eta \in (0,1]$ such as, for all $t\in[a,b]$,  and all $\ep>0$, $ \Omega^t_{\ep/\eta}\subset\mathcal{A}_t(\Omega_\ep) \subset \Omega^t_{\eta \ep}$. 
\item Uniformly in $t\in [a,b]$, $\ds\mu_d(\Omega^t\setminus\mathcal{A}_t(\Omega_\ep)) \conv{\ep}{0} 0$, where $\mu_d$ denotes the $d$-dimensional Lebesgue measure.
\end{enumerate}
\end{proposition}

\begin{proof}
\begin{enumerate}
\item This is true because $\mathcal{A}_t$ is a homeomorphism of $\R^d$. 
\item Thanks to Assumption \ref{ass:A1}, the family of diffeomorphisms $(\mathcal{A}_t)_{t\in[a,b]}$ is uniformly bilipschitz on a bounded neighborhood of $\Omega$. In particular, we have the existence of a positive constant $K\geq 1$ such as, for all $t\in[a,b]$ and all $\xx,\yy\in\Omega$
\begin{align*}
\frac{1}{K}|\xx-\yy| \leq |\mathcal{A}_t(\xx)-\mathcal{A}_t(\yy)| \leq K |\xx-\yy|,
\end{align*}
and one checks that $\eta:=1/K$ does the job.
\item By bijectivity, we have $\Omega^t\setminus \mathcal{A}_t(\Omega_\ep) = \mathcal{A}_t(\Omega \setminus \Omega_\ep)$. Recall estimate \eqref{ineq:jacob}.\\
We hence have $\mu_d(\Omega^t\setminus \mathcal{A}_t(\Omega_\ep)) \leq \beta \mu_d(\Omega\setminus \Omega_\ep)$ which goes to $0$ with $\ep$, by inner regularity of the Lebesgue measure. $\qedhere$
\end{enumerate}
\end{proof}

We end this subsection with the following time-dependent counterpart of  Proposition \ref{propo:poinc}

\begin{proposition}\label{propo:poinc:t} Consider $\gamma$ the positive number defined in Proposition \ref{propo:tube}. Then, for $\ep\in[0,\gamma)$, and $t\in[a,b]$, the open sets $\mathcal{A}_t(\Omega_\ep)$ share a common Poincaré-Wirtinger constant $\Cc_{\Omega,\gamma}^{\mathcal{A}}$ (in the sense of Proposition \ref{propo:poinc}).
\end{proposition}
\begin{proof}
Fix $t\in[a,b]$, $\ep\in[0,\gamma)$ and $u\in\H^1(\mathcal{A}_t(\Omega_\ep))$. Since $\mathcal{A}_t$ is bilipschitz, $v:=u\circ\mathcal{A}_t \in \H^1(\Omega_\ep)$. Thanks to Proposition \ref{propo:poinc} we have 
\begin{align*}
\int_{\Omega_\ep} |v(\yy)|^2\, \dd \yy \leq C_{\Omega,\gamma}^2 \int_{\Omega_\ep} |\nabla v(\yy)|^2 \, \dd\yy.
\end{align*}
Considering the change of variable $\yy=\mathcal{A}_t^{-1}(\xx)$ in the previous inequality, together with estimate \eqref{ineq:jacob} we get
\begin{align*}
\frac{1}{\beta}\int_{\mathcal{A}_t(\Omega_\ep)} |u(\xx)|^2\, \dd \xx \leq \frac{C_{\Omega,\gamma}^2}{\alpha} \int_{\mathcal{A}_t(\Omega_\ep)} |(\nabla v\circ\mathcal{A}_t^{-1})(\xx)|^2 \, \dd\xx.
\end{align*}
But from $v=u\circ\mathcal{A}_t$ we deduce for all $\yy\in\Omega_\ep$, $|\nabla v (\yy)| \leq \|\nabla \mathcal{A}_t\|_{\Ld^\infty(\Omega)} |(\nabla u \circ\mathcal{A}_t)|(\yy) $. Because of Assumption \ref{ass:A1}, we have $\|\nabla \mathcal{A}_t\|_{\Ld^\infty(\Omega)}\leq \|\nabla_{\xx}\Theta\|_{\Ld^\infty([a,b]\times\Omega)}$, so that eventually 
\begin{align*}
\|u\|_{\Ld^2(\mathcal{A}_t(\Omega_\ep))} \leq \sqrt{\frac{\beta}{\alpha}} C_{\Omega,\gamma}\|\nabla_{\xx}\Theta\|_{\Ld^\infty([a,b]\times\Omega)} \|\nabla u\|_{\Ld^2(\mathcal{A}_t(\Omega_\ep)},
\end{align*}
and the proof is over taking $\Cc_{\Omega,\gamma}^{\mathcal{A}} := \beta^{1/2}\alpha^{-1/2} C_{\Omega,\gamma}\|\nabla_{\xx}\Theta\|_{\Ld^\infty([a,b]\times\Omega)}$. $\qedhere$
\end{proof}

\subsection{The easy case}\label{subsec:easy}

 The assumptions of Theorem \ref{thm:aubin:gen} below are the direct generalization of the usual framework in which the Aubin-Lions Lemma is frequently invoked : strong estimate (positive Sobolev) for the space variable, weak estimate (negative Sobolev) for the time variable. Since the domain is non-cylindrical, we translate the weak assumption into its variationnal formulation \eqref{ineq:lem:gen}. Though this result may clearly be generalized, we intentionnally state it in a rather particular case, to show the simplicity of the argument used.

\begin{theorem}\label{thm:aubin:gen} Let $1 \leq p < \infty$ and $(f_n)_n$ a sequence of functions such as  $(f_n)_n \dot{\,\in\,} \textnormal{L}^p(\hat{\Omega})$ and $(\nabla_\xx f_n)_n \dot{\,\in\,} \textnormal{L}^p(\hat{\Omega})$. We assume the existence of a constant $\Cc>0$ and an integer $N\in\N$ such as, for any test function $\psi \in\mathscr{D}(\hat{\Omega})$,
\begin{align}
\label{ineq:lem:gen}\left|\langle  \partial_{t} f_n, \psi\rangle \right| \leq \Cc \sum_{|\alpha|\leq N} \|\partial_{\xx}^\alpha\psi\|_{\textnormal{L}^2(\hat{\Omega})}.
\end{align}
Then, for any $\ep>0$, one has $(f_n)_n\ddot{\,\in\,}\Ld^p(\hat{\Omega}_\ep)$.
\end{theorem}
\begin{rmk}
The proof below applies both in the cylindrical and the non-cylindral case. The statement would still be true  replacing the assumption on the gradient by some uniform equicontinuity in the $\xx$ variable. In fact in this latter form, we recover Kruzhkov's lemma (see \cite{kru} and also \cite{and} for a more recent presentation). The proof of Kruzhkov used also the uniform approximation by convolution, but Kruzhkov used it to obtain then uniform equicontinuity in the time variable, and the proof below is a bit different. However we will use Kruzhkov's strategy in the more intricate case of Navier-Stokes' estimates (see subsection \ref{subsec:nav}).
\end{rmk}
\begin{rmk}\label{rmk:ass}
This proof does not use the full strength of Assumption \ref{ass:A1}, the only important thing is that $\Omegah$ is connected and bounded. Reproducing the arguments below in a similar way, we could replace $\Omegah$ by any open set $O\subset\R^m_{\yy}\times\R^d_{\xx}$, and assumption \eqref{ineq:lem:gen} by a similar one on $(\nabla_{\yy} f_n)_n$, and get $(f_n)_n\ddin\Ll^p(O)$.
\end{rmk} 
\begin{proof}
Recall the positive number $\eta>0$ defined in point 2. of Proposition \ref{propo:tubular}. We only need to prove that for all integer $m\in\N$, $(f_n)_n\ddin\textnormal{L}^p(\Omegah_{1/m})$ : the conclusion will follow by a standard diagonal extraction, since $\Omegah_\ep \subset\Omegah_{1/m}$ as soon as $1/m < \ep$.

\vspace{2mm}
We recall the sequence $(\ffi_\ell)_\ell$  of nonnegative even mollifiers (in space only) : $\ffi_\ell(\xx) := \ell^{d}\ffi( \ell \xx)$, where $\ffi$ is some smooth even nonnegative function with support in the unit ball of $\R^d$. In the sequel the convolution $\star$ has to be understood in the space variable $\xx$ only.

\vspace{2mm}

  For all $t$,  $\mathcal{A}_t(\Omega_{1/m})\subset\Omega_{\eta /m}^t$  (see the definition of $\eta$ in Proposition \ref{propo:tubular}). Since $f_n$ is only defined in $\Omegah$, $f_n \star\ffi_\ell$ is well-defined only in a subset of $\Omegah$. Typically, if $\ell \geq 2m/\eta$ (and we will assume this from now on),  $f_n\star\ffi_\ell$ is well-defined in $\Omegah_{1/m}$. In that case,  for any $\psi\in\D(\Omegah_{1/m})$, $\psi\star\ffi_\ell \in\D(\Omegah)$. Since $(\nabla_{\xx} f_n)_n\din\textnormal{L}^p(\Omegah)$ we have the standard estimate 
\begin{align}
\label{limunif}\lim_{\ell \rightarrow +\infty} \sup_{n\in\N} \|f_n - f_n\star\rho_{\ell} \|_{\textnormal{L}^p(\Omegah_{1/m})} = 0.
\end{align}
Now fix $\ell\geq 2m/\eta$.

Since $(f_n)_n\din\textnormal{L}^p(\Omegah)$, we have $(f_n\star\ffi_\ell)_n \din\textnormal{L}^p(\Omegah_{1/m})$ and $(\nabla_{\xx} f_n\star\ffi_\ell)_n \din\textnormal{L}^p(\Omegah_{1/m})$. For the time derivative we just write for all $\psi\in\mathscr{D}(\Omegah_{1/m})$, using the fact that $\ffi_\ell$ is even,
\begin{align*}
\langle \partial_t (f_n \star\ffi_\ell) , \psi \rangle = \langle \partial_t f_n, \psi\star \ffi_\ell \rangle.
\end{align*}
Now (since $\ell$ is large enough), we have $\ffi_{\ell}\star\psi \in\D(\Omegah)$ and it is hence an admissible test-function for the estimate \eqref{ineq:lem:gen}. Eventually, for any $\psi\in\mathscr{D}(\Omegah_{1/m})$, we have
\begin{align*}
\left|\langle  \partial_{t} (f_n \star \ffi_\ell), \psi\rangle \right| &\leq \Cc \sum_{|\alpha|\leq N} \|\partial_{\xx}^\alpha(\ffi_\ell\star \psi)\|_{\textnormal{L}^2(\hat{\Omega})}\\
&\leq \Cc_{\ffi_\ell} \|\psi\|_{\textnormal{L}^1(\hat{\Omega})},
\end{align*}
from which we deduce by duality that $(\partial_t(f_n\star\ffi_\ell))_n \din\textnormal{L}^{\infty}(\Omegah_{1/m})$. We in particular obtain that, for any fixed $\ell$, ${(f_n \star\ffi_\ell)_n \dot{\,\in\,} \W^{1,p}(\Omegah_{1/m})}$ whence $(f_n \star\ffi_\ell)_n \ddot{\,\in\,} \textnormal{L}^{p}(\Omegah_{1/m})$ by Rellich-Kondrachov's theorem
We then extract diagonally (without reindexing) in order to have, for any $\ell$, the convergence of $(f_n\star\ffi_\ell)_n$, in $\textnormal{L}^p(\Omegah_{1/m})$. Now we may conclude by writting for any $n,q\in\N$
\begin{align}
\label{resum:preuve}f_n -f_q = (f_n -f_n\star\ffi_\ell)  + (f_n\star\ffi_\ell -f_q\star\ffi_\ell) + (f_q\star\ffi_\ell -f_q).
\end{align}
When $\ell$ is fixed, since $(f_n\star\ffi_\ell)_n$ is converging in $\textnormal{L}^p(\Omegah_{1/m})$, the second term of the r.h.s. goes to $0$ in $\textnormal{L}^p(\Omegah_{1/m})$ when $\min(n,q)\rightarrow +\infty$. The two other ones go both to $0$ in the same space as $\ell\rightarrow +\infty$, uniformly in $n,q$ thanks to \eqref{limunif}. The considered extraction is hence a Cauchy sequence in $\textnormal{L}^p(\Omegah_{1/m})$, and converges in this space.   $\qedhere$
\end{proof}

The missing convergence in $\Ld^p(\Omegah)$, concentrated on the peel $\Omegah\setminus\Omegah_\ep$, can be in fact recovered : all we need is some $\Ld^p$-equi-continuity of the sequence $(f_n)_n$ on the sets $\Omegah\setminus\Omegah_\ep$. In order to do so, we use the regularity Assumption \ref{ass:A1} through estimate \eqref{ineq:sobot} and Proposition \ref{propo:tubular} proven above. We then have the following local to global Proposition
\begin{proposition}\label{propo:l2g}
Fix $p\in[1,\infty[$. Assume that $(f_n)_n$,  $(\nabla_{\xx} f_n)_n$ are both bounded in $\Ld^p(\Omegah)$. If for all $\ep>0$, one has $(f_n)_n\ddin\textnormal{L}^p(\Omegah_\ep)$ (local compactness), then $(f_n)_n\ddin\Ld^p(\Omegah)$ (global compactness).
\end{proposition}
\begin{proof}
Assume $p<d$.

\vspace{2mm}

Since $(f_n)_n \din\textnormal{L}^p(\Omegah)$ and $(\nabla_{\xx} f_n)_n\din\textnormal{L}^p(\Omegah)$, we hence deduce from estimate \eqref{ineq:sobot} the following bound 
\begin{align}
\label{ineq:bornps}\sup_{n\in\N} \int_a^b \|f_n(t)\|_{\textnormal{L}^{p^\star}(\Omega^t)}^p \dd t < \infty. 
\end{align}
Now for $\ep>0 $, by Hölder inequality
\begin{align*}
\|f_n\|_{\textnormal{L}^p(\Omegah \setminus \Omegah_{\ep})}^p = \int_a^b \|f_n(t)\|_{\textnormal{L}^p(\Omega^t\backslash \mathcal{A}_t(\Omega_\ep))}^p \dd t \leq \int_a^b \|f_n(t)\|_{\textnormal{L}^{p^\star}(\Omega^t)}^p \mu_d (\Omega^t\setminus \mathcal{A}_t(\Omega_\ep))^{1-\frac{p}{p^\star}} \dd t,
\end{align*}
where $\mu_d$ is the $d$-dimensionnal Lebesgue measure, so that thanks to \eqref{ineq:bornps} and point 3. of Proposition \ref{propo:tubular}, one eventually have
\begin{align}
\label{ineq:pel}\sup_{n\in\N} \|f_n\|_{\Ld^p(\Omegah\setminus\Omegah_\ep)} \conv{\ep}{0} 0.
\end{align}
Since for each $\ep>0$, we have $(f_n)_n\ddin\Ld^p(\Omegah_\ep)$, by diagonal extraction we have the existence a subsequence of $(f_n)_n$ converging in each $\Ld^p(\Omegah_\ep)$, and \eqref{ineq:pel}   allows to see that such a subsequence is in fact converging in $\Ld^p(\Omegah)$. 

\vspace{2mm}

The case $p\geq d$ is completely similar, replacing $p^\star$ by $p+1$ in the proof above.
$\qedhere$ \end{proof}

Using the previous Proposition, one deduces the following Corollary for Theorem  \ref{thm:aubin:gen}.
\begin{coro}\label{coro:aubin:gen}
Under the assumptions of Theorem \ref{thm:aubin:gen},  we have in fact $(f_n)_n \ddin\textnormal{L}^p(\Omegah)$.
\end{coro}

\subsection{Divergence-free vector fields}\label{subsec:div}
In $\Ld^2(\Omega)$, since $\mathscr{D}(\Omega)$ is a dense subset, one may recover the norm of an element $\uu$ of $\Ld^2(\Omega)$ by the duality formula 
\begin{align}
\label{eq:dua}\|\uu\|_2 = \sup_{\fffi\in\D(\Omega), \|\fffi\|_2 \leq 1} \left\langle \uu, \fffi\right\rangle_{\textnormal{L}^2(\Omega)}.
\end{align}
Now recall the notations $\Ddiv(\Omega)$, $\Ldiv^2(\Omega)$, $\Ldivoo^2(\Omega)$ and $\gamma_n$ given in subsection \ref{subsec:not}. $\Ddiv(\Omega)$ is not dense $\Ldiv^2(\Omega)$ (its closure is $\Ldivoo^2(\Omega)$). This has for consequence that one may not expect a duality formula as \eqref{eq:dua}, when testing only against divergence-free smooth functions. As a matter of fact, there is a dual estimate of the same flavor for $\Ldiv^2(\Omega)$, but one has to take into account the normal trace. The following result is proven in the appendix section \ref{sec:appe}.
\begin{lem}\label{lem:norm2} Denote by $\Cc_\Omega$  the Poincaré-Wirtinger constant of $\Omega$. For all $\uu\in\Ldiv^2(\Omega)$ one has
\begin{align}
\label{ineq:lem:norm2}\| \uu\|_{2} \leq \sup_{\fffi\in\Ddiv(\Omega), \|\fffi\|_2 \leq 1} \left\langle \uu, \fffi\right\rangle_{\textnormal{L}^2(\Omega)} + (1+\Cc_\Omega)\|\gamma_n\uu\|_{\H^{-1/2}(\partial \Omega)}.
\end{align}
\end{lem}
\begin{rmk}\label{rmk:normeq}
Together with \eqref{ineq:trace:div} (continuity of $\gamma_n$), Lemma \ref{lem:norm2} allows to see that the r.h.s. of \eqref{ineq:lem:norm2} defines a norm which is equivalent to the $\|\cdot\|_2$ norm on $\Ldiv^2(\Omega)$. Hence \eqref{ineq:lem:norm2} is a generalization of \eqref{eq:dua}.
\end{rmk}
The previous results may of course be adapted to the case of divergence-free (in $\xx$) vector fields defined on $\Omegah$. Recall the notation given in subsection \ref{subsec:not} for $\Ddiv(\Omegah)$, $\Ldiv^2(\Omegah)$ and $\Ldivoo^2(\Omegah)$. It is possible to define a normal trace on $\Ldiv^2(\Omega)$ even though it is a bit tedious. Indeed, thanks to \eqref{ineq:trace:div}, for all $t\in (a,b)$ the operator norm of  $\gamma_{n}^t:\Ldiv^2(\Omega^t)\rightarrow\H^{-1/2}(\partial\Omega^t)$ is not greater than $1$, so that one has, for all $\fffi \in\Ddiv(\R\times\R^d)$, 
\begin{align*}
 \int_a^b \|\gamma_{n}^t\fffi(t)\|_{\H^{-1/2}(\Omega^t)}^2 \dd t
&\leq \int_a^b \|\fffi(t)\|_{\textnormal{L}^2(\Omega^t)}^2 \dd t = \|\fffi\|_{\textnormal{L}^2(\Omegah)}^2,
\end{align*}
which allows to define a ``normal'' trace operator on $\Ldiv^2(\Omegah)$, the quotes refering to the fact that the normal vector is here orthogonal to $\partial\Omega^t$ in $\R^d$ and not to $\partial\Omegah$ in $\R^{d+1}$. This normal trace lies in the space denoted $\H^{-1/2}_{\xx}(\partial\Omegah)$, defined as the completion of  $\mathscr{C}^\infty(\partial\Omegah)$ for the norm
\begin{align*}
\|\psi\|_{\H_{\xx}^{-1/2}(\partial\Omegah)}:=\left( \int_a^b \|\psi(t)\|_{\H^{-1/2}(\partial \Omega^t)}^2 \dd t\right)^{1/2}.
\end{align*}
This normal trace operator will be denoted by $\gamma_{\widehat{n}}:\Ldiv^2(\Omegah)\rightarrow \H_{\xx}^{-1/2}(\partial\Omegah)$. 

\vspace{2mm}

Since for $\delta<\gamma$, $\Omega_\delta$ is Lipschitz (see Proposition \ref{propo:tube}), so are the diffeomorphic domains $\mathcal{A}_t(\Omega_\delta)$. We can  hence reproduce the previous analysis for $\Omegah_\delta$  and define similarly a normal trace on $\partial\Omegah_\delta$, that we denote $\gamma_{\widehat{n},\delta}:\Ldiv^2(\Omegah_\delta)\rightarrow\H_{\xx}^{-1/2}(\partial\Omegah_\delta)$.

\vspace{2mm}

 As in the  stationnary case, $\Ldiv^2(\Omegah_\delta)$ vector fields have a dual estimate for their norm. More precisely, using Lemma \ref{lem:norm2} to get estimate \eqref{ineq:lem:norm2} at each time $t$ and Proposition \ref{propo:poinc:t} to handle the dependency of the Poincaré-Wirtinger constant w.r.t. to $t\in[a,b]$ and $\delta\in[0,\gamma)$, we get, after integration in time 
\begin{lem}\label{lem:norm2t}  
Recall the definition of $\gamma$ in Proposition \ref{propo:tube}. For all $\delta\in[0,\gamma)$ and all $\uu\in\textnormal{L}^2(\Omegah_\delta)$ 
\begin{align*}
\| \uu\|_{\textnormal{L}^2(\Omegah_\delta)} \leq \sup_{\ppsi\in\Ddiv(\Omegah_\delta), \|\ppsi\|_{2} \leq 1} \prods{\uu}{\ppsi}{\textnormal{L}^2(\Omegah)} + (\Cc_{\Omega,\gamma}^\mathcal{A}+1)\|\gamma_{\widehat{n},\delta}\uu\|_{\H_{\xx}^{-1/2}(\partial\Omegah_\delta )},
\end{align*}
where $\Cc_{\Omega,\gamma}^\mathcal{A}$ is the constant of Proposition \ref{propo:poinc:t}.\end{lem}

\subsection{Proof of Theorem \ref{thm:aubin}}\label{subsec:nav}
At first sight, Theorem \ref{thm:aubin} seems to be a particular case ($p=2$) of Theorem \ref{thm:aubin:gen}  (and Proposition \ref{propo:l2g} for the global compactness). Indeed, $(\uu_n)_n\din\Ld^2(\Omegah)$, $(\nabla_{\xx}\uu_n)_n\din\Ld^2(\Omegah)$ and  assumptions \eqref{ineq:lem} resembles \eqref{ineq:lem:gen}. However, the two assumptions do not merge: in \eqref{ineq:lem:gen}, the estimate is known against every test function, whereas \eqref{ineq:lem} involves only divergence-free test functions. In fact, if one tries to reproduce the nice and simple proof of Theorem  \ref{thm:aubin:gen}, there is an issue when trying to get the time compactness for the convolution product $(\uu_n\star\ffi_{\ell})_n$. The obstruction is in fact present in both the cylindrical and the non-cylindrical case and may be understood without the time variable. For all $\uu\in\Ld^2(\Omega)$, let
\begin{align*}
N(\uu):= \sup_{\|\ppsi\|_2 \leq 1, \ppsi\in\Dd(\Omega)}  \langle \uu,\ppsi\rangle.
\end{align*}
The prickle is here : \emph{$N$ is only a semi-norm}, this is a direct consequence of Remark \ref{rmk:normeq}. The issue is clear: with an estimate such as \eqref{ineq:lem} only one part of the $\textnormal{L}^2$ norm of $(\partial_t \uu_n \star\rho_\ell)_n$ will be controlled, the other one being handled by the normal trace on the boundary. Filling this gap is the main difficutly in the proof below.


\begin{proof}[Proof of Theorem \ref{thm:aubin}]
Instead of using  Proposition \ref{propo:l2g} like we did in Corollary \ref{coro:aubin:gen}, we will prove a bound for $(\uu_n)_n$ in $\Ld^r(\Omegah)$ with $r>2$, which will give us some $\Ld^2$-equi-continuity on all $\Omegah$ and  allow us to work with time-translation to get compactness. 

\vspace{2mm}

Since $(\uu_n)_n$ and $(\nabla_{\xx}\uu_n)_n$ are both bounded in $\Ld^2(\Omegah)$, we have, using estimate \eqref{ineq:sobot}
\begin{align}
\label{ineq:bornps2} \sup_{n\in\N} \int_a^b \|\uu_n(t)\|_{\textnormal{L}^{q}(\Omega^t)}^2 \dd t < \infty,
\end{align}
for some $q>2$. Now pick $r\in]2,q[$, so that $1/r=\theta/q + (1-\theta)/2$ with $r\theta \leq 2$.  Using the interpolation $\Ld^r = [\Ld^2,\Ld^q]_\theta$ together with $(\uu_n)_n\din\Ld^\infty(\R;\Ld^2(\R^d))$ and \eqref{ineq:bornps2}, we get $(\uu_n)_n\din\Ld^r(\Omegah)$. The compactness $(\uu_n)_n\ddin\textnormal{L}^2(\Omegah)$ may hence be deduced from the weaker one $(\uu_n)_n\ddin\Ll^2(\Omegah)$. 

\vspace{2mm}

Since $(\nabla_{\xx} \uu_n)_n\din\textnormal{L}^2(\Omegah)$, thanks to Riesz-Fréchet-Kolmogorov's Theorem we only need to prove local uniform equi-continuity in time 
\begin{align*}
(\lambda_s \uu_n - \uu_n)_s \conv{s}{0}^{\Ll^2(\Omegah)} 0,\text{ uniformly in }n,
\end{align*}
where $\lambda_s$ is the time-translation operator $\lambda_s \uu_n(t,\xx) := \uu_n(t-s,\xx)$. 

\vspace{2mm}

 Point 2. of  Proposition \ref{propo:tube} expresses that, though $\mathcal{A}_t(\Omega_\delta)$ is not the $\delta$-interior $\Omega^t=\mathcal{A}_t(\Omega)$, it may be framed between two such sets namely $\Omega^t_{\ep/\eta}$ and $\Omega^t_{\eta \ep}$. To simplify the presentation of the proof, we will assume from now on that $\eta=1$, so that $\mathcal{A}_t(\Omega_\delta) = \Omega^t_\delta$. The general case follows adapting line to line the proof below, the core of the argument remaining intact. 

\vspace{2mm}

First notice that since $\uu_n$ belongs to $\Ld^2(\R;\Ldiv^2(\R^d))$ and vanishes outisde $\Omegah$, it easy to check that $\gamma_{\widehat{n}}\uu_n=0$. However, for $\delta\in(0,\gamma)$, nothing can be said about $\gamma_{\widehat{n},\delta}\uu_n$. This is why we introduce the orthogonal projection
\begin{align*}
\Pp_\delta : \Ldiv^2(\Omegah_\delta)\rightarrow \Ldivoo^2(\Omegah_\delta).
\end{align*}
As before, $\ffi_\delta$ is the mollifier (in space only) $\ffi_\delta(\xx):=\delta^{-d}\ffi(\delta^{-1}\xx)$, with $\ffi\in\D(\R^d)$ a nonnegative even function with a support in the unit ball (and integral $1$). Notice that since $\uu_n$ vanishes outside $\Omegah$, $\uu_n\star\ffi_\delta$ vanishes outside 
\begin{align*}
\Omegah_{-\delta}:= \bigcup_{a<t<b} \{t\} \times \Omega_{-\delta}^t,
\end{align*}
where we recall the convention $\Omega_{\delta}^t = \Omega^t + B(0,\delta)$.
\begin{itemize}
\item[\emph{Step 1.}]
Because of Lemma \ref{lem:norm2t}, for $\delta < \gamma/2$ one has
\begin{align*}
\|\uu_n\star\ffi_{\delta} - \Pp_{2\delta} (\uu_n \star \ffi_{\delta})\|_{\textnormal{L}^2(\Omegah_{2\delta})} \leq (\Cc_\Omega^\mathcal{A}+1) \|\gamma_{\widehat{n},2\delta} (\uu_n \star\ffi_{\delta})\|_{\H_{\xx}^{-1/2}(\partial\Omegah_{2\delta})}.
\end{align*}
Since the normal trace operator $\gamma_{\widehat{n},2\delta}$ has a norm not greater than $1$, and since $\uu_n\star\ffi_\delta$ vanishes outside $\Omegah_{-\delta}$, one gets, using Hölder inequality for the last line,
\begin{align*}
\|\uu_n\star\ffi_{\delta} - \Pp_{2\delta} (\uu_n \star \ffi_{\delta})\|_{\textnormal{L}^2(\Omegah_{2\delta})} &\leq (\Cc_\Omega^\mathcal{A}+1) \|\uu_n \star\ffi_{\delta}\|_{\textnormal{L}^2(\Omegah_{-\delta} \setminus \Omegah_{2\delta})}\\
&\leq (\Cc_\Omega^\mathcal{A}+1) \|\uu_n \|_{\textnormal{L}^2(\Omegah_{-2\delta} \setminus \Omegah_{3\delta})}\\
&\leq (\Cc_\Omega^\mathcal{A}+1) \mu_{d+1}(\Omegah_{-2\delta} \setminus \Omegah_{3\delta})^{\frac 12- \frac 1r} \|\uu_n\|_{\textnormal{L}^{r}(\Omegah)},
\end{align*}
where $\mu_{d+1}$ denotes the $(d+1)$-dimensionnal Lebesgue measure and $r>2$ has been chosen previously so that $(\uu_n)_n\din\Ld^r(\Omegah)$. Since $\ds \mu_{d+1}(\Omegah_{-2\delta} \setminus \Omegah_{3\delta}) \conv{\delta}{0} 0$ (by dominated convergence),  we infer 
\begin{align}
\label{conv:step1}\sup_{n\in\N} \|\uu_n\star\ffi_{\delta} - \Pp_{2\delta} (\uu_n \star \ffi_{\delta})\|_{\textnormal{L}^2(\Omegah_{2\delta})}\conv{\delta}{0} 0.
\end{align}
\item[\emph{Step 2.}] 
Let's prove that for all $\delta>0$, there exists $\xi>0$ such as 
\begin{align*}
\forall \psi \in \D(\Omegah_{2\delta}), \quad 0< \sigma \leq \xi  \Longrightarrow \lambda_{-\sigma} \psi \in \D(\Omegah_\delta).
\end{align*}

\vspace{2mm}

Fix $\delta>0$. Because of Assumption \ref{ass:A1} and some uniform continuity argument, we have the existence of $\xi>0$ such as for all $0<\sigma\leq \xi$ and all $t \in[a,b]$, $\mathcal{A}_{t+\sigma}(\Omega_{2\delta}) \subset \mathcal{A}_t(\Omega_\delta)$. Now consider $\psi\in\D(\Omegah_{2\delta})$. If $(t,\xx) \notin\Omegah_\delta$, then $\xx\notin\mathcal{A}_t(\Omega_\delta)$, whence $\xx\notin\mathcal{A}_{t+\sigma}(\Omega_{2\delta})$ because of the previous inclusion. We eventually get $(t+\sigma,\xx)\notin \Omegah_{2\delta}$, that is, by assumption, $\psi(t+\sigma,\xx)=0$ which means also  $\lambda_{-\sigma} \psi(t,\xx) =0$ and we have proved $\text{Supp}\,\lambda_{-\sigma} \psi \subset\Omegah_\delta$.

\item[\emph{Step 3.}] 
Estimate \eqref{ineq:lem} together with the standard properties of the convolution operator allows to prove, for any \emph{fixed} $\delta>0$, the existence of some positive constant  $\Cc_\delta>0$ such as
\begin{align}
\label{ineq:lemep}
\forall \ppsi\in\Ddiv(\Omegah_\delta),\quad \langle \uu_n \star \ffi_{\delta} , \ppsi \rangle = \langle \uu_n , \ppsi\star \ffi_{\delta}  \rangle  \leq \Cc_\delta \| \ppsi\|_2.
\end{align}

For any pair $(\vv,\Pphi)\in\D(\R\times\R^3)^2$, we have 
\begin{align*}
\langle \lambda_{s} \vv-\vv,\boldsymbol{\Phi}\rangle =  s\int_0^1 \langle \vv, \lambda_{-s z}\partial_t \boldsymbol{\Phi}\rangle \,\dd z,
\end{align*}
where the bracket is simply the inner product of  $\textnormal{L}^2(\R\times\R^d)$. By density, we may use this formula with  $\vv = \uu_n\star\ffi_{\delta}$ and $\boldsymbol{\Phi} \in\Ddiv(\Omegah_{2\delta})$. But because of Step 2. we know that for $s$ small enough and any $z\in[0,1]$, for any $\psi\in\Ddiv(\Omegah_{2\delta})$, one has $\lambda_{-sz} \psi \in\Ddiv(\Omegah_\delta)$. Estimate \eqref{ineq:lemep} is hence usable with the test function $\ppsi := \lambda_{-sz} \partial_t \boldsymbol{\Phi}$. Using the duality formula of Lemma  \ref{lem:norm2t} we eventally infer, for any fixed  $\delta$, 
\begin{align}\label{ineq:step3}
\sup_{n\in\N}\| (\lambda_s-\Id)\Pp_{2\delta}(\uu_n\star\ffi_\delta)\|_{\textnormal{L}^2(\Omegah_{2\delta})} \conv{s}{0} 0.
\end{align}
\item[\emph{Step 4.}] 
Now we may conclude the proof. As mentionned, it suffices to prove that for all compact subset $K$ of $\Omegah$,
\begin{align*}
(\lambda_s \uu_n - \uu_n)_s \conv{s}{0}^{\Ld^2(K)} 0,\text{ uniformly in }n.
\end{align*}
For this purpose, we write, for any $\delta\in(0,\gamma/2)$
\begin{align*}
\lambda_s \uu_n -\uu_n &=  (\lambda_s -\Id)(\uu_n-\uu_n\star\ffi_{\delta}) \\
& + (\lambda_s-\Id)(\uu_n\star\ffi_{\delta}-\Pp_{2\delta}\uu_n\star\ffi_{\delta})\\
& + (\lambda_s-\Id)\Pp_{2\delta}\uu_n\star\ffi_{\delta}.
\end{align*}
and we proceed line by line, in the $\textnormal{L}^2(\Omegah_{\ep})$ norm. We take $\delta$ small enough so that  $K \subset \Omegah_{2\delta}$ and in particular $\|\cdot\|_{\Ld^2(K)} \leq \|\cdot \|_{\Ld^2(\Omega_{2\delta})}$.

We may first pick $\delta$ small enough to treat the first line of the r.h.s. thanks to the bound on the gradient in the space variable $(\nabla_{\xx} \uu_n)_n\din\textnormal{L}^2(\Omegah)$ and the usual estimate 
\begin{align*}
\| \uu_n-\uu_n\star\ffi_\delta \|_{\textnormal{L}^2(K)} \leq \| \uu_n-\uu_n\star\ffi_\delta \|_{\textnormal{L}^2(\Omegah_{2\delta})} \leq  \delta \|\nabla_{\xx} \uu_n\|_{\textnormal{L}^2(\Omegah)}.
\end{align*}
The first line of the r.h.s. may hence be made arbitrarily small with $\delta$, independently  of $n$ (and $s$). The uniform (in $n$) convergence \eqref{conv:step1} of Step 1. allows to treat the second line of the r.h.s. in the same way.  Once $\delta$ has been fixed to handle the previous lines, the third one is lastly handled thanks to estimate \eqref{ineq:step3} above in Step 3. $\qedhere$

\end{itemize}
\end{proof}

\begin{coro}\label{coro:thm:aubin}
Consider $(\uu_n)_n\din\Ldiv^2(\Omegah)$ such as $(\nabla_{\xx} \uu_n)_n \din\Ld^2(\Omegah)$ and $(\mathds{1}_{\Omegah}\uu_n)_n\din\Ld^\infty(\R;\Ld^2(\R^d))$. Asume furthermore that $(\uu_n)_n$ satisfies the dual estimate \eqref{ineq:lem}  and $(\gamma_{\widehat{n}} \uu_n)_n\ddin\H^{-1/2}_{\xx}(\partial\Omegah)$. Then $(\uu_n)_n\ddin\Ld^2(\Omegah)$.
\end{coro}
\begin{rmk}
The main difference with Theorem \ref{thm:aubin} is that $\uu_n$ here is not defined outside $\Omegah$. Extending it  by $0$ outside is always possible, but then nothing insures that $\uu_n\in\Ld^2(\R;\Ldiv(\R^3))$, since its normal trace is \emph{a priori} not zero. 
\end{rmk}
\begin{rmk}
This Corollary is designed for nonhomogeneous boundary conditions. It includes the case of an approximation procedure in which the boundary condition is fixed independantly of $n$ (in which case $(\gamma_{\widehat{n}}\uu_n)_n$ is obviously compact) and the case of a penalization strategy in which one forces the boundary condition to hold at the limit by imposing some behavior outside $\Omegah$ (see \cite{fuj} for instance).
\end{rmk}
\begin{proof}
The beginning of the proof is the same  as in Theorem \ref{thm:aubin}, so that we only need to prove for every compact subset $K$ of $\Omegah$
\begin{align*}
(\lambda_s \uu_n - \uu_n)_s \conv{s}{0}^{\Ld^2(K)} 0,\text{ uniformly in }n.
\end{align*}
We extract from $(\gammah \uu_n)_n$ a converging subsequence (that we don't relabel) and consider $(\ww_n)_n$ a corresponding sequence of lifts such as $(\ww_n)_n$ is convergent in $\Ldiv^2(\Omegah)$. As before, we write 
\begin{align*}
(\lambda_s -\Id)\uu_n &= (\lambda_s -\Id)(\uu_n-\uu_n\star\ffi_{\delta}) \\
& + (\lambda_s-\Id)(\uu_n\star\ffi_{\delta}-\Pp_{2\delta}\uu_n\star\ffi_{\delta})\\
& + (\lambda_s-\Id)\Pp_{2\delta}\uu_n\star\ffi_{\delta}.
\end{align*}
In the r.h.s. of the previous equality the first and last lines are handled exactly as before. For the second line, we may write 
\begin{align*}
(\lambda_s-\Id)(\uu_n\star\ffi_{\delta}-\Pp_{2\delta}\uu_n\star\ffi_{\delta}) = (\lambda_s-\Id)(\zz_n\star\ffi_\delta-\Pp_{2\delta}\zz_n\star\ffi_\delta) + (\lambda_s-\Id)(\ww_n-\Pp_{2\delta}\ww_n), 
\end{align*}
where $\zz_n:=\uu_n-\ww_n$. When $\delta$ is fixed such as $K\subset\Omegah_{2\delta}$, the second term of the r.h.s. is going to $0$ in $\Ld^2(K)$ as $s\rightarrow 0$ (uniformly in $n$) simply because $(\ww_n-\Pp_{2\delta}\ww_n)_n$ is converging in $\Ld^2(\Omegah_{2\delta})$ (whence in particular relatively compact) : indeed, $(\ww_n)_n$ is converging in $\Ld^2(\Omegah)$, and $\Id-\Pp_{2\delta}$ is a continuous endomorphism of $\Ld^2(\Omegah_{2\delta})$ . As for the first term in the r.h.s., it suffices  to prove
\begin{align*}
\sup_{n\in\N} \|\zz_n\star\ffi_{\delta} -\Pp_{2\delta}\zz_n\star\ffi_\delta\|_{\Ld^2(\Omegah_{2\delta})} \conv{\delta}{0} 0.
\end{align*} 
This limit can be proven reproducing exactly Step 1. of the proof of Theorem \ref{thm:aubin}, because $\gamma_{\widehat{n}}\zz_n = 0$, which allows to extend $\zz_n$ by $0$ outside $\Omegah$, and have $(\zz_n)_n\din\Ld^2(\R;\Ldiv^2(\R^d))$ (and so, recover the assumption on $\uu_n$ in Theorem \ref{thm:aubin}). $\qedhere$
\end{proof}

\section{Appendix}\label{sec:appe}
We give here the proof, Lemma \ref{lem:norm2} and Proposition \ref{propo:poinc}. 
\begin{proof}[Proof of Lemma \ref{lem:norm2}]
 Consider the orthogonal projection ${\Pp:\Ldiv^2(\Omega)\rightarrow \Ldivoo^2(\Omega)}$ and $\ww \in\Ddiv(\R^d)$. Since $\Ldivoo^2(\Omega)=\overline{\Ddiv(\Omega)}^{\|\cdot\|_2}$, we have  
\begin{align*}
\|\Pp\ww\|_2 &= \sup_{\fffi\in\Ddiv(\Omega), \|\fffi\|_2 \leq 1} \int_\Omega \Pp\ww(\xx)\cdot \fffi(\xx) \dd\xx \\
&= \sup_{\fffi\in\Ddiv(\Omega), \|\fffi\|_2 \leq 1} \int_\Omega \ww(\xx)\cdot \fffi(\xx) \dd\xx.
\end{align*}
It remains to estimate $\|\ww-\Pp \ww\|_2$ for which one can solve the Neumann-Laplacian problem
\begin{align*}
 \Delta v &= 0, \text{ on } \Omega,\\
 \int_\Omega v &= 0,\\
 \partial_{\nn} v &= \gamma_n \ww \text{ on } \partial \Omega.
\end{align*}
The previous boundary problem is well posed and the variational formulation gives directly, $\|\nabla v\|_{2}^2 \leq \| \gamma_n\ww \|_{\H^{-1/2}(\partial \Omega)} \|\gamma v\|_{\H^{1/2}(\Omega)}$, whence 
\begin{align*}
\|\nabla v\|_2^2 \leq \|\gamma_n\ww\|_{\H^{-1/2}(\partial\Omega)} \|v\|_{\H^1(\Omega)} \leq \sqrt{1+\Cc_\Omega^2} \|\gamma_n\ww\|_{\H^{-1/2}(\partial\Omega)} \|\nabla v\|_{2},
\end{align*}
so that eventually $\|\nabla v\|_2^2 \leq (1+\Cc_\Omega) \|\gamma_n\ww\|_{\H^{-1/2}(\partial\Omega)}$. We hence have ${\ww-\nabla v} \in \Ldivoo^2(\Omega)$ and since the orthogonal projection minimizes the distance we get
\begin{align*}
\| \ww -\Pp \ww \|_2 \leq (1+\Cc_\Omega)
 \|\gamma_n \ww \|_{\H^{-1/2}(\Omega)}. 
\end{align*}
The desired estimate \eqref{ineq:lem:norm2} follows then directly for $\ww\in\Ddiv(\R^d)$, and we may extend it by density to $\uu \in\Ldiv^2(\Omega)$. $\qedhere$\end{proof}

\begin{proof}[Proof of Proposition \ref{propo:poinc}]
We recall (see Theorem 5' in Chap. 6 of \cite{stein} for instance) that, for a Lipschitz domain $\Omega$ one has the existence of a continuous linear extension operator $\P_\Omega$, mapping $\W^{1,p}(\Omega)$ to $\W^{1,p}(\R^d)$, whose norm depends solely on the Lipschitz constant of $\partial\Omega$. 

\vspace{2mm}

Thanks to Proposition \ref{propo:tube}, we hence have the existence of extension operators $\P_{\Omega_\ep}:\H^1(\Omega_\ep)\rightarrow\H^1(\R^d)$, parametrized by $\ep\in[0,\gamma)$, and such as the norms of all these operators are all bounded by some constant depending only on $\gamma$, say $M_\gamma$.

\vspace{2mm} 

Let's proceed as in the usual proof for the Poincaré inequality and argue by contradiction. The opposite statement would imply the existence of a sequence $(\ep_n)_n \in[0,\gamma)$,  and a sequence $u_n\in\H^1(\Omega_{\ep_n})$, such as 
\begin{align*}
\|u_n\|_{\Ld^2(\Omega_{\ep_n})}=1,\qquad\| \nabla u_n\|_{\Ld^2(\Omega{\ep_n})} \leq 1/n,\qquad\int_{\Omega_{\ep_n}} u_n(\xx)\,\dd \xx = 0.
\end{align*}
Without loss of generality, we may assume that $(\ep_n)_n$ is monotonically to some $\ep \in[0,\gamma]$. 

\vspace{2mm}

 Now, because of the previous bounds on $(u_n)_n$, and since the extension operators $\P_{\Omega_{\ep_n}}$ are all of norm less than $M_\gamma$, the corresponding sequence of extensions $v_n:=\P_{\Omega_{\ep_n}} u_n$ satisfies $(v_n)_n\din\H^1(\R^d)$. In particular, thanks to  Rellich-Kondrachov's Theorem, we have, $(v_n)_n\ddin\Ll^2(\R^d)$, so that, up to an extraction (that we don't write), we have $(v_n)_n$ converging to some $v$ in $\Ll^2(\R^d)$. 

\vspace{2mm}

Since $v_n$ equals $u_n$ on $\Omega_{\ep_n}$, we have that $(\|\nabla v_n\|_{\Ld^2(\Omega_{\ep_n})})_n$ converges to $0$. But the sequence $(\ep_n)_n$ is monotone, so that either all $\Omega_{\ep_n}$ are included in $\Omega_\ep$ either they all contain this open set. It both case, $(\nabla v_n)_n$ hence converges to $0$ in $\mathscr{D}'(\Omega_{\ep})$ (at least) : $\nabla v = 0$ and $v$ is constant on $\Omega_{\ep}$. But on the other hand 
\begin{align*}
v\mathds{1}_{\Omega_\ep} - v_n\mathds{1}_{\Omega_{\ep_n}} = (v-v_n)\mathds{1}_{\Omega_{\ep_n}} + v (\mathds{1}_{\Omega_{\ep}}-\mathds{1}_{\Omega_{\ep_n}}) \conv{n}{+\infty}^{\Ll^2(\R^d)} 0,
\end{align*} 
where we used the $\Ll^2(\R^d)$  convergence of $(v_n)_n$ to treat  the first term in the r.h.s. and the dominated convergence Theorem to treat the second one. In particular, if $K$ is some bounded set containing $\Omega_\ep$ and all the open sets $\Omega_{\ep_n}$, $(v_n\mathds{1}_{\Omega_{\ep_n}})_n$ converges to $v\mathds{1}_{\Omega_\ep}$ in $\Ld^2(K)$, so that the $\Ld^2(\Omega_\ep)$ norm of $v$ equals $1$ and the mean-value of $v$ on $\Omega_\ep$ vanishes, but this impossible if $v$ is contant on $\Omega_\ep$. $\qedhere$
\end{proof}

\section*{Acknowledgements}
The research leading to this paper was (partially) funded by the french ``ANR blanche'' project Kibord:  ANR-13-BS01-0004. The author would also like to warmly thank Céline Grandmont and François Murat. Céline Grandmont pointed out numerous difficulties (and mistakes) arising in subection \ref{subsec:nav} and François Murat has always been available for fruitful discussions. Eventually, the author thanks also Boris Andreianov and Clément Cancès who gave him several valuable remarks.

\bibliography{biblio}

\end{document}